\documentclass[leqno,11pt]{amsart}
\usepackage{amssymb, amsmath}
\usepackage{amsthm, amsfonts,mathrsfs}
\usepackage{amsmath,amssymb,amsthm,amsfonts,mathrsfs}
\usepackage{fancyhdr}
\usepackage{amsfonts,graphicx}

\def\dv{\mbox{div}}
\def\rot{\mbox{rot}}
\def\Supp{\mathop{\rm Supp}\nolimits\ }

\theoremstyle{plain}
\newtheorem{theorem}{Theorem}[section]
\newtheorem{lemma}{Lemma}[section]
\newtheorem{corollary}{Corollary}[section]
\newtheorem{definition}{Definition}[section]
\newtheorem{proposition}{Proposition}[section]
\newtheorem{Rema}{Remark}[section]

\newcommand{\ZZ}{\mathbb{Z}}  
  
\newcommand{\NN}{\mathbb{N}}

\newcommand{\RR}{\mathbb{R}}

\date{}

\title[ Helicoidal solution for the Euler system]
{Global well-posedness of helicoidal Euler equations}
\author[H.  Abidi]{Hammadi Abidi}
\address[H.  Abidi]{D\'epartement de Math\'ematiques
Facult\'e des Sciences de Tunis Campus universitaire 2092 Tunis,
Tunisia} \email{habidi@univ-evry.fr}
\author[S. Sakrani]{Saoussen Sakrani} \address[S. Sakrani]{D\'epartement de Math\'ematiques
Facult\'e des Sciences de Tunis Campus universitaire 2092 Tunis,
Tunisia} \email{saoussensakrani2012@gmail.com}

\begin{document}
\maketitle

\begin{abstract}
This paper deals with the global existence and uniqueness results for the three-dimensional incompressible Euler equations with a particular structure  for initial data lying in critical spaces. In this case the BKM  criterion is not known.
\end{abstract}

\,

{\bf{AMS Subject Classifications}}~: 76D03 (35B33 35Q35 76D05)\\
{\bf{Keywords}}~: Helicoidal Euler equation; Global existence; Uniqueness.

\section{Introduction}
The purpose of this paper is to investigate the global well-posedness of the following three-dimensional incompressible Euler system in the whole space with helicoidal initial data. This system is described as follows:
$$
{\rm(E)}\quad
\left\{\begin{array}{l}
\displaystyle\partial_t u +(u\cdot\nabla) u +\nabla\Pi =0,\qquad (t,x)\in\RR^+\times\RR^3,\\
\displaystyle{\mathop{\rm div}}\,u =0, \\
\displaystyle u_{|t=0}=u^0.
\end{array}
\right.
$$
Here,  the vector field $u=(u_1,u_2,u_3)$ is the velocity of the fluid and $\Pi$ is a scalar pressure function.
The operator $u.\nabla$ is given explicitly by $u.\nabla =\displaystyle{\sum_{j=1}^3} u_j\partial_j$ and the incompressibility of the fluid is expressed via the second equation of the system ${\mathop{\rm div}}\,u=\displaystyle{\sum_{j=1}^3} \partial_ju_j=0$.\\
The question of local or global existence and uniqueness of solutions to the system $(\rm{E})$ is one of the most important problems in fluid mechanics.
Existence and uniqueness theories of ($2$ or $3$ dimensional) Euler equations have been studied by many mathematicians and physicists. W. Wolibner \cite{WO} started the subject in H\"older spaces, D. Ebin \cite{ebin}, J. Bourguignon \cite{bourguignon}, R. Temam \cite{temam}, T. Kato and G. Ponce \cite{ponce} worked out this subject in Sobolev spaces. Much of the studies on the Euler equations of an ideal incompressible fluid in Besov spaces has been done by M. Vishik (\cite{vishik}, \cite{vishik1}, \cite{vishik2}), D. Chae \cite{chae} and C. Park and J. Park \cite{park} . \\
The question of global existence (even for smooth initial data) is still open and continues
to be one of the most challenging problems in nonlinear PDEs. The degree of
difficulties depends strongly on the dimensions (2 or 3) and the regularity of the initial
data. In this context, the vorticity play a fundamental role. In fact, the
well-known BKM criterion \cite{majda} ensures that the development of finite time singularities
for Kato's solutions is related to the blowup of the $L^{\infty}$ norm of the vorticity near
the maximal time existence. In $2$-D, the vorticity satisfies a transport equation
$$
\partial_t \omega +(u\cdot\nabla) \omega=0.
$$
In space dimension three, the vorticity satisfies the equation
\begin{equation}\label{TR}
\partial_t \omega +(u\cdot\nabla) \omega=(\omega\cdot\nabla)u
\end{equation}
and the main difficulty for establishing global regularity is to understand how the
vortex-stretching term $(\omega\cdot\nabla)u$ affects the dynamic of the fluid. While the question of global existence for $3$-D Euler system is widely open, some
positive results are available for the $3$-D flows with some geometry constraints as the
so-called axisymmetric flows without swirl.
 We say that a vector field u is axisymmetric if it has the form :
$$
u(x, t) = u_r(r, z , t)e_r + u_z (r, z , t)e_z,\quad x=(x_{1},x_{2},z),\quad r=({x_{1}^2+x_{2}^2})^{\frac12},
$$ 
where  $\big(e_r, e_{\theta} , e_z\big)$ is the cylindrical basis of $\RR^3$ and the components $u_r$ and $u_z$ do not depend on the angular variable.  
The main feature of axisymmetric flows arises in the vorticity which takes the form
$$
\omega=(\partial_zu^r-\partial_ru^z) e_{\theta},
$$
 and satisfies 
  \begin{equation}
 \label{tourbillon}
\partial_t \omega +(u\cdot\nabla)\omega =\frac{u^r}{r}\omega.  
\end{equation}
Consequently the quantity $\alpha:=\omega/r$ satisfies
\begin{equation}
\label{equation_importante}
\partial_t\alpha+(u\cdot\nabla)\alpha=0,
\end{equation}
which induces the conservation of all its $L^ p$ norms for every $p\in[1,\infty].$ 
Ukhovskii and Yudovich \cite{ivdovich} took advantage of these conservation laws to
prove the global existence for axisymmetric initial data with finite energy and satisfying
in addition $\omega^{0}\in L^2\cap L^\infty$ and $\frac{\omega^0}{r}\in L^2\cap L^\infty.$ In terms of Sobolev regularity these assumptions are satisfied if the velocity $u_{0}\in H^s$ with $s>{7\over 2}$.   This is far from critical regularity of local existence theory $s=\frac52.$ The optimal result in Sobolev spaces is done by Shirota and Yanagisawa \cite{shirota} who proved global existence in $H^s$, with  $s>\frac52.$ 
In a recent work,  R. Danchin \cite{rd} has weakened the Ukhoviskii and Yudovich conditions.   More precisely, he obtain the global existence and uniqueness for initial data $\omega^0\in L^{3,1}\cap L^\infty$ and $\frac{\omega^0}{r}\in L^{3,1}.$ 
Recently, in \cite{taoufik} the first author and his collaborators proved the global existence to the system $(\rm{E})$ for initial data $u^0\in B_{p,1}^{1+\frac{3}{p}}$ and $\frac{\omega^0}{r}\in L^{3,1}.$
\;

In the same context (i.e geometric contraints), Dutrifoy was interested in this question and he was published several papers, in \cite{dutrifoy}, he proved global existence to the incompressible Euler equations with a particular geometric structure, the focus is on so-called helicoidal solutions. In \cite{rd}, Danchin proved too global existence for helicoidal initial data and the aim in this paper is to prescribe regularity conditions on the vorticity.
\begin{definition}\label{he}
Let $k$ be a nonnegative real number. We say that a vector field $u=u_re_r+u_{\theta} e_{\theta}+u_ze_z$ is helicoidal  if:\\
1)The components $u_r,u_{\theta}$ et $u_z$ of $u$ are constant on helicoids $z=z_0+k\theta$ et $r=r_0$.\\
2) At every point of $\RR^3$ the vector field $u$ is orthogonal to $h:=re_{\theta}+ke_z.$
\end{definition}
We note that the limit case  $k= 0$ corresponds to the definition of an axisymmetric vector field.\\
The main characteristic of helical flows is the vorticity takes the following form:
$$
k\omega =h\omega_z\qquad\mbox{and}\qquad
\partial_t\omega_z+(u\cdot\nabla)\omega_z=0
$$
where $\omega_z$ is the vertical component of the vorticity. 
Thus
$$
\frac{|\omega(t,\psi(t,x))|}{\sqrt{k^2+r^2(\psi(t,x))}}=\frac{|\omega(0,x)|}{\sqrt{k^2+r^2}}
\quad\mbox{where $\psi$ is the flow associated to  velocity $u$}.
$$
In this paper we shall not be interested in the dependence with regard to $k$ quantities to be measured, and we shall thus suppose to simplify that $k=1.$
Our main result in this paper is concerning the unique solvability of ${\rm(E)}$ with the initial data helicoidal in the critical Besov spaces
(for the definition see the next section). Here and in what follows, we shall always denote $(1,x,y)f=(f,xf,yf).$ More precisely we obtain the following result.
\begin{theorem}\label{globale}
 Let $u^0$ be an helicoidal  divergence free vector field with 
 $(1,x,y)u^0\in L^2(\RR^2\times]-\pi,\pi[),$ such that its vorticity satisfies 
$\omega^0\in\mathscr{\dot B}^{0}_{2,1}(\RR^2\times]-\pi,\pi[),$ 
$(1,x,y)\omega^0\in\mathscr{\dot B}^{0}_{\infty,1}(\RR^2\times]-\pi,\pi[)$ and
$\omega^0_z\in\mathscr{\dot B}^{0}_{1,1}(\RR^2\times]-\pi,\pi[).$  Then the system (E) has a unique global solution $(1,x,y)u\in L^\infty_{loc}(\RR_+;\,L^2)$ such that
$\omega\in L^\infty_{loc}(\RR_+;\,\mathscr{\dot B}^{0}_{2,1}),$
$(1,x,y)\omega\in L^\infty_{loc}(\RR_+;\,\mathscr{\dot B}^{0}_{\infty,1})$ and
$\omega_z\in L^\infty_{loc}(\RR_+;\,\mathscr{\dot B}^{0}_{1,1}).$
Moreover for every $t\in\RR_+$
$$
\|(1,x,y)u(t)\|_{L^2}\le C_0e^{C_0t}
$$
and
$$
||\omega_z(t)||_{\mathscr{\dot B}^{0}_{1,1}}+
||\omega(t)||_{\mathscr{\dot B}^{0}_{2,1}}
+||(1,x,y)\omega(t)||_{\mathscr{\dot B}^{0}_{\infty,1}}\leq C_0 e^{exp(e^{C_0t})}
$$
where $C_0$ depends on the norms of $u^0.$
\end{theorem}
\begin{Rema}
According to \cite{rd},  $u^0\in L^2(\RR^2\times]-\pi,\pi[)$ can be replaced by
$\omega^0\in L^{2,1}(\RR^2\times]-\pi,\pi[).$ 
\end{Rema}
\noindent{\bf Scheme of the proof and organization of the paper.} 
The main difficulty is the proof of Theorem \ref{globale} lies in the fact that when the initial data belongs to critical spaces, we can not use  the Beale-Kato-Majda  criterion. Thus, we owe controlled $\|\nabla u\|_{L^\infty(\RR^2\times]-\pi,\pi[)},$ which is bounded by
$\sum_{n\in\ZZ}\|\omega_n\|_{\dot B^0_{\infty,1}(\RR^2)}$ (where $\omega_n$ is the Fourier coefficients of $\omega$ see Lemma \ref{LP}). 
For that we shall rewrite \eqref{TR} (see Corollary \ref{C1})
$$
\partial_t\omega+(\widetilde{u}\cdot\nabla_h)\omega=\left(\begin{array}{c}-\omega_z
\widetilde u_2 \\ \omega_z\widetilde u_ 1\\0\end{array}\right),
$$
where we denote
$$
\widetilde{u}=(u_1+yu_3,u_2-xu_3),\quad \nabla_h=(\partial_x,\partial_y)\quad
\mbox{and}\quad
\omega=(\omega_1,\omega_2,\omega_z)=(-y\omega_z,x\omega_z,\omega_z).
$$
Motivated by \cite{taoufik, hmidi}, for some $n\in\ZZ,$ let $\tilde\omega_{n}$ 
solves the following system
$$
\left\{\begin{array}{l}
\displaystyle\partial_t \tilde\omega_{1,n} +(\tilde u\cdot\nabla_h)\tilde\omega_{1,n} 
 =-\tilde\omega_{z,n}\tilde u_2,\\
\displaystyle\partial_t \tilde\omega_{2,n} +(\tilde u\cdot\nabla_h)\tilde\omega_{2,n} 
 =\tilde\omega_{z,n}\tilde u_1,\\
\displaystyle\partial_t\tilde\omega_{z,n} +(\tilde u\cdot\nabla_h)\tilde\omega_{z,n} =0,\\
\tilde \omega_{n|t=0}=\tilde \omega_{n}(0)
\end{array} 
\right.
$$
where 
$$
\tilde \omega_{n}(0)=\begin{pmatrix}-y\omega^0_{n,z}\\ 
x\omega^0_{n,z}\\ 
\omega^0_{n,z}\\ 
\end{pmatrix},
\qquad
\qquad\mbox{and}\qquad
\omega^0_{n,z}
=\frac{1}{2\pi}\int_{-\pi}^{\pi}\omega^0_{z}e^{-inz}dz
=\partial_xu^0_{n,2}-\partial_yu^0_{n,1}.
$$
By Proposition \ref{Fourier}, we deduce that $\tilde\omega_{n}$ is the Fourier coefficients of $\omega,$ i.e, $\tilde\omega_{n}=\omega_{n}.$ Thus
$
\omega=\sum_{n\in\ZZ}\omega_{n}e^{inz}.
$
Finally to control $\|\omega_n\|_{\dot B^0_{\infty,1}},$ we will use a new approach similar to  \cite{hmidi}, which consists to  linearize properly the Fourier  of transport equation. For that, we will localize in frequency the initial data and denote by $\tilde\omega_{q,n}$   the unique global vector-valued solution of the problem 
$$
\left\{\begin{array}{l}
\displaystyle\partial_t \tilde\omega_{1,q,n} +(\tilde u\cdot\nabla_h)\tilde\omega_{1,q,n} 
 =-\tilde\omega_{z,q,n}\tilde u_2,\\
\displaystyle\partial_t \tilde\omega_{2,q,n} +(\tilde u\cdot\nabla_h)\tilde\omega_{2,q,n} 
 =\tilde\omega_{z,q,n}\tilde u_1,\\
\displaystyle\partial_t\tilde\omega_{z,q,n} +(\tilde u\cdot\nabla_h)\tilde\omega_{z,q,n} =0,\\
\tilde \omega_{q,n|t=0}=\tilde \omega_{q,n}(0)
\end{array} 
\right.
$$
where 
$$
\tilde \omega_{q,n}(0)=\begin{pmatrix}-y\dot \Delta_q \omega^0_{n,z}\\ 
x\dot \Delta_q \omega^0_{n,z}\\ 
\dot \Delta_q \omega^0_{n,z}\\ 
\end{pmatrix},
\qquad
\mbox{and}\qquad
\dot \Delta_q \omega^0_{n,z}
=\frac{1}{2\pi}\int_{-\pi}^{\pi}\dot \Delta_q \omega^0_{z}e^{-inz}dz
=\partial_x\dot\Delta_qu^0_{n,2}-\partial_y\dot\Delta_qu^0_{n,1}.
$$
In the second section, we shall collect some basic facts on
Littlewood-Paley analysis;
then  in section 3 is devoted to the study of some geometric properties
of any solution to a vorticity equation model; finally in the last section, \mbox{we prove Theorem \ref{globale}.}
\medbreak \noindent{\bf Notations:}  Let $A, B$ be two operators, we
denote $[A,B]=AB-BA,$ the commutator between $A$ and $B$. For
$a\lesssim b$, we mean that there is a uniform constant $C,$ which
may be different on different lines, such that $a\leq Cb$. For $X$ a Banach space and $I$ an interval of $\RR,$ we denote by
${\mathscr{C}}(I;\,X)$ the set of continuous functions on $I$ with
values in $X.$  For $q\in[1,+\infty],$ the
notation $L^q(I;\,X)$ stands for the set of measurable functions on
$I$ with values in $X,$ such that $t\longmapsto\|f(t)\|_{X}$ belongs
to $L^q(I).$
We always denote the Fourier transform of a function $u$ by
$\hat{u}$ or $\mathcal{F}(u).$
\section{The functional tool box}

The proof  of Theorem \ref{globale} requires a dyadic decomposition
of the Fourier variables, or Littlewood-Paley decomposition (see \cite{bhouri}).  Let $\varphi\in\mathcal{S}(\RR^2)$  be smooth function supported in 
$\mathcal{C}= \{\xi\in\RR^2,\frac{3}{4}\leq|\xi|\leq\frac{8}{3}\}$  such that
\begin{equation*}
 \sum_{q\in\ZZ}\varphi(2^{-q}\xi)=1 \quad\hbox{for}\quad \xi\neq 0.
\end{equation*}
For every $u\in{\mathcal S}'(\RR^2)$ one defines the homogeneous 
Littlewood-Paley operators by 
$$
\forall q\in\ZZ,\quad \dot\Delta_qu=\varphi(2^{-q}\textnormal{D})
u\hspace{1cm}\mbox{and} \hspace{1cm} \dot S_qu=\sum_{j\leq
q-1}\dot\Delta_{j}u.
$$
We notice that these operators can be written as a convolution.   
For example for $q\in\ZZ,\,$ \mbox{$\dot\Delta_{q}u=2^{2q}h(2^q\cdot)\star u,$}
 where $h\in\mathcal{S}$ and $\widehat{h}(\xi)=\varphi(\xi). $\\
 We have the formal decomposition
\begin{equation*}
u=\sum_{q\in\ZZ}\dot\Delta_q\,u,\quad\forall\,u\in {\mathcal
{S}}'(\RR^2)/{\mathcal{P}}[\RR^2],
\end{equation*}
where ${\mathcal{P}}[\RR^2]$ is the set of polynomials  (see \cite{PE}). Moreover, the Littlewood-Paley decomposition satisfies
the property of almost orthogonality:
\begin{equation}\label{Pres_orth}
\dot\Delta_k\dot\Delta_q u\equiv 0 \quad\mbox{if}\quad| k-q|\geq 2
\quad\mbox{and}\quad\dot\Delta_k(\dot S_{q-1}u\dot\Delta_q u) \equiv
0\quad\mbox{if}\quad| k-q|\geq 5.
\end{equation}
We recall now the definition of homogeneous Besov type spaces from
\cite{bhouri}.
\begin{definition}\label{Besov}
{\sl  Let $(p,r)\in[1,+\infty]^2,$ $s\in\RR$ and $u\in{\mathcal
S}'(\RR^2),$ we set
$$
\|u\|_{\dot B^s_{p,r}}=\Big(2^{qs}\|\dot\Delta_q
u\|_{L^{p}}\Big)_{\ell ^{r}}.
$$
\begin{itemize}

\item
For $s<\frac{2}{p}$ (or $s=\frac{2}{p}$ if $r=1$), we  define $\dot
B^s_{p,r}(\RR^2)=\big\{u\in{\mathcal S}'(\RR^2)\;\big|\; \Vert
u\Vert_{\dot B^s_{p,r}}<\infty\big\}.$

\item
If $k\in\NN$ and $\frac{2}{p}+k\leq s<\frac{2}{p}+k+1$ (or
$s=\frac{2}{p}+k+1$ if $r=1$), then $\dot B^s_{p,r}(\RR^2)$ is
defined as the subset of distributions $u\in{\mathcal S}'(\RR^2)$
such that $\partial^\beta u\in\dot B^{s-k}_{p,r}(\RR^2)$ whenever
$|\beta|=k.$
\end{itemize}}
\end{definition}

\begin{Rema}\label{rmk1.1}
\begin{enumerate}
  \item We point out that if $s>0$ then $B^s_{p,r}=\dot B^s_{p,r}\cap L^p$ and
$$
\|u\|_{B^s_{p,r}}\approx \|u\|_{\dot B^s_{p,r}}+\|u\|_{L^p}
$$
with $B^s_{p,r}$ being the non-homogeneous Besov space.
  \item It is easy to verify that the homogeneous Besov
space $\dot{B}^s_{2,2}(\RR^2)$ coincides with the classical
homogeneous Sobolev space $\dot{H}^{s}(\RR^2)$ and
$\dot{B}^s_{\infty,\infty}(\RR^2)$ coincides with the classical
homogeneous H\"older space $\dot{C}^s(\RR^2)$  when $s$ is not
positive integer, in case $s$ is a nonnegative integer,
$\dot{B}^s_{\infty,\infty}(\RR^2)$ coincides with the classical
homogeneous Zygmund space $\dot{C}^s_{\ast}(\RR^2).$
  \item Let $s\in \mathbb{R}, 1\le p,r\le\infty$, and $u \in
{\mathcal S}'(\RR^2).$ Then $u$ belongs to $\dot{B}^{s}_{p, r}(\RR^2)$ if and
only if there exists $\{c_{j, r}\}_{j \in \mathbb{Z}} $ such that
$\|c_{j, r}\|_{\ell^{r}} =1$ and
\begin{equation*}
\|\dot{\Delta}_{j}u\|_{L^{p}}\leq C c_{j, r} \, 2^{-j s }
\|u\|_{\dot{B}^{s}_{p, r}}\qquad \mbox{for all}\ \ j\in\ZZ.
\end{equation*}
\end{enumerate}
\end{Rema}
For the convenience of the reader,  we recall some basic facts on
Littlewood-Paley theory, one may check \cite{bhouri} for more
details.
\begin{lemma}\label{lb} Let ${\mathcal B}$ be a ball and ${\mathcal C}$ an  annulus of $\RR^N.$
 A constant $C$ exists so
that for any positive real number $\delta,$ any non-negative integer
$k,$ any smooth homogeneous function $\sigma$ of degree $m,$ and any
couple of real numbers $(a, \; b)$ with $ b \geq a \geq 1,$ there
hold
$$
\begin{aligned}
&\Supp \hat{u} \subset \delta\mathcal{B} \Rightarrow \sup_{|\alpha|=k}
\|\partial^{\alpha} u\|_{L^{b}} \leq  C^{k+1}
\delta^{k+ N(\frac{1}{a}-\frac{1}{b} )}\|u\|_{L^{a}},\\
& \Supp \hat{u} \subset \delta \mathcal{C} \Rightarrow C^{-1-k}\delta^{
k}\|u\|_{L^{a}}\leq \sup_{|\alpha|=k}\|\partial^{\alpha} u\|_{L^{a}}
\leq
C^{1+k}\delta^{ k}\|u\|_{L^{a}},\\
& \Supp \hat{u} \subset \delta\mathcal{C} \Rightarrow \|\sigma(D)
u\|_{L^{b}}\leq C_{\sigma, m} 
\delta^{ m+N(\frac{1}{a}-\frac{1}{b})}\|u\|_{L^{a}}.
\end{aligned}
$$
\end{lemma}
In what follows, we shall frequently use Bony's decomposition
\cite{bony} in the both homogeneous and inhomogeneous context:
$$
\begin{aligned}
uv=\dot{T}_u v+\dot{R}(u,v)=\dot{T}_u v+\dot{T}_v u+\dot{\mathcal{R}}(u,v)
\end{aligned}
$$
where
$$
\begin{aligned}
&\dot{T}_u v=\sum_{q \in\ZZ}\dot S^h_{q-1}u\dot\Delta^h_q
v,\qquad
\dot{R}(u,v)=\sum_{q\in\ZZ}\dot\Delta_q u \dot S_{q+2}v,\\
&\dot{\mathcal{R}}(u,v)=\sum_{q\in\ZZ}\dot\Delta_q u
\widetilde{\dot\Delta}_{q}v\quad  \mbox{with}\quad
\widetilde{\dot\Delta}_{q}v= \sum_{|q'-q|\leq
1}\dot\Delta_{q'}v.
\end{aligned}
$$
\begin{definition}
Let $u$ be a mean free function in $\mathcal{S}'(\RR^2\times]-\pi,\pi[),$ 
$2\pi$-periodic with respect the third variable, 
$(p,r)\in[1,+\infty]^2$ and $s\in\RR$ be given real numbers.
Then $u$ belongs to the Besov space $\mathscr{\dot B}^s_{p,r}$
if and only if
$$
\|u\|_{\mathscr{\dot B}^s_{p,r}}=\sum_{n\in\ZZ}\|u_n\|_{\dot B^s_{p,r}}<+\infty
$$
where $u_n$ is the Fourier coefficients are computed as follows
$$
u_n=\frac{1}{2\pi}\int_{-\pi}^{\pi}u(\cdot,\cdot,z)e^{-inz}dz.
$$
\end{definition}
\begin{Rema}
Let $(p,r)\in[1,+\infty]^2$ and $s\in\RR,$ then for all 
$u\in\mathcal{S}'(\RR^2\times]-\pi,\pi[),$ $2\pi$-periodic with regard to the third 
variable, we have
$$
\|u\|_{\dot B^s_{p,r}(\RR^2)}
\lesssim
\|u\|_{\mathscr{\dot B}^s_{p,r}}.
$$
\end{Rema}
\section{Geometric properties of the vorticity}
\begin{proposition}\label{wassim}
Let $u=(u_1,u_2,u_z)$  be a smooth helicoidal vector field.  Then \\
the vector $\omega=\nabla \times u=(\omega_1,\omega_2,\omega_z)$ satisfies for every $(x_1,x_2,z)\in \RR^3$,
$$
x_1\omega_1(x_1,x_2,z)+x_2\omega_2(x_1,x_2,z)=0.
$$
\end{proposition}

\begin{proof}
 In the cylindrical coordinate system, we have
 \begin{equation}\label{rot}
 \omega=\nabla \times u=\left(\begin{array}{c}\frac{1}{r}\partial_\theta u_z-\partial_z u_\theta \\\partial_zu_r-\partial_ru_z \\ \frac{u_\theta}{r}+\partial_ru_\theta
 -\frac{1}{r}\partial_\theta u_r\end{array}\right),
\end{equation}
 and the second point of the Definition \ref{he}, implies that
 $$
 \omega=r\omega_ze_{\theta}+\omega_ze_z.
 $$
Then
$$
\omega^1(x_1,x_2,z)=-x_2\omega_z
$$
and
$$
\omega^2(x_1,x_2,z)=x_1\omega_z.
$$
Therefore
$$
x_1\omega_1(x_1,x_2,z)+x_2\omega_2(x_1,x_2,z)=0.
$$
This achieves the proof.
\end{proof}
\begin{proposition}\label{nabla}
Let $u$  be an helicoidal vector field. Then
$$
\big|\partial_r(\frac{u_{\theta}}{r})\big|+\big|\partial_z(\frac{u_{\theta}}{r})\big|
\lesssim
|\partial^2_xu|+|\partial^2_yu|+|\partial^2_{xy}u|+|\nabla u|.
$$
\end{proposition}

\begin{proof}
According to the inequality \eqref{rot}, we have
$$
\frac{u_\theta}{r}=\omega_z-\partial_ru_\theta
 +\frac{1}{r}\partial_\theta u_r
 $$
 where $\omega_z$ is the vertical component of  $\rot\,u.$ One has
 $$
 \partial_r=\cos(\theta)\partial_x+\sin(\theta)\partial_y
 $$
 and
 $$
 \frac{1}{r}\partial_\theta=-\sin(\theta)\partial_x+\cos(\theta)\partial_y
 $$
 it follows that 
 $$
 \partial_r(\frac{1}{r}\partial_\theta)=-\sin(\theta)\cos(\theta)\partial^2_x
 +\sin(\theta)\cos(\theta)\partial^2_y+\bigl(\cos^2(\theta)-\sin^2(\theta)\bigr)\partial^2_{xy}
 $$
 and
 $$
  \partial^2_r=\cos^2(\theta)\partial^2_x+\cos^2(\theta)\partial^2_y
  +2\cos(\theta)\cos(\theta)\partial^2_{xy}.
 $$
 Thus we find
 $$
\big|\partial_r(\frac{u_{\theta}}{r})\big|
\lesssim
|\partial^2_xu|+|\partial^2_yu|+|\partial^2_{xy}u|.
$$
Since $u$ is helicoidal, then
$$
\partial_z(\frac{u_{\theta}}{r})=-\frac{1}{r}\partial_\theta u_{\theta}
$$
thus
$$
\big|\partial_z(\frac{u_{\theta}}{r})\big|
\lesssim
|\nabla u|.
$$
This finishes the proof of Proposition.
\end{proof}

$\bullet$ The last part of this section is dedicated to the study of a vorticity equation type in
which no relations between the vector field u and the solution $\Omega$ are supposed. More
precisely, we consider
\begin{equation}\label{bassem}
\left\{ \begin{array}{l}
\displaystyle\partial_t \Omega + u.\nabla \Omega= \Omega .\nabla u,\\
\displaystyle\dv\, u=0\\
\displaystyle\Omega _{|t=0}=\Omega^0.
\end{array} \right.
\end{equation}
\begin{proposition}\label{mahdi}
Let $u$ be a divergence free and helicoidal vector field  such that $\nabla u$ and $\nabla^2u$ belonging to
$ L^1_{loc}(\RR_+;\,L^\infty(\RR^3))$ and $\Omega$ the unique global solution of \eqref{bassem}  with smooth initial data $\Omega^0$. Then, the following properties hold. \\
i) If $\dv\,\Omega^0=0,$ then $\dv\,\Omega(t)=0,$ for every $t\in \RR_+$.\\
ii) If $\Omega^0=r\Omega^0_z e_{\theta}+\Omega^0_z e_z,$ then we have
$$
\Omega(t)=r\Omega_z(t) e_{\theta}+\Omega_z(t) e_z,\quad \forall t\in \RR_+.
$$
Consequently, $\Omega(t,x_1,0,z)=\Omega(t,0,x_2,z)=0$ and
$$
\partial_t \Omega+(u\cdot\nabla)\Omega=\Omega_z(u_re_{\theta}-u_{\theta}e_r).
$$
\end{proposition}
\begin{proof}
First, we notice that the existence and uniqueness of global solution can be
done in classical way. Indeed, let $\psi$ the flow of the velocity $u$,
$$
\psi(t,x)=x+\int_0^t u(\tau,\psi(\tau,x))d\tau.
$$
Since $u\in L^1_{loc}(\RR_+;\, Lip(\RR^3))$  then it follows from the ODE theory that the function $\psi$ is uniquely and globally defined.\\
Let $\tilde \Omega(t,x):= \Omega(t,\psi(t,x))$ and $A(t,x)$ the matrix such that 
$A(t,\psi^{-1}(t,x))={(\partial_j u_i)}_{1\leq i,j\leq 3}$.\\
It's clear that 
$$
\partial_t \tilde \Omega=A(t,x)\tilde \Omega.
$$
From Cauchy$- $Lipschitz theorem this last equation has a unique global solution, and
the system \eqref{bassem} too.\\
i) We apply the divergence operator to the equation \eqref{bassem} , leading under the assumption $\dv\, u=0,$ to
$$
\partial_t \dv\, \Omega +u\cdot\nabla \dv\, \Omega=0.
$$
Then, the quantity $\dv\, \Omega$ is transported by the flow and consequently the incompressibility of $\Omega$ remains true for every time.\\
ii) We have
$$
(u\cdot\nabla \Omega)\cdot e_r=u.\nabla \Omega_r-\frac{1}{r}u_{\theta}\Omega_{\theta}
$$
and
$$
(\Omega\cdot\nabla u)\cdot e_r=\Omega\cdot\nabla u_r-\frac{1}{r}\Omega_{\theta}u_{\theta}
$$
then the component $\Omega_r,$ verifies
\begin{eqnarray}
\partial_t  \Omega_r+u\cdot\nabla \Omega_r= \Omega\cdot\nabla u_r
=\Omega_r\partial_r u_r+(\Omega_z-\frac{\Omega_{\theta}}{r})\partial_z u_r.
\end{eqnarray}
From the maximum principle we deduce
$$
||\Omega_r(t)||_{L^{\infty}}
\leq 
\int_0^t\bigl(||\Omega_r(\tau)||_{L^{\infty}}
+\big\|(\Omega_z-\frac{\Omega_{\theta}}{r})(\tau)\big\|_{L^{\infty}}\bigr)
\|\nabla u(\tau)\|_{L^{\infty}}d\tau.
$$
The component $\Omega_{\theta}$ satisfies the following equation 
$$
\partial_t \Omega_{\theta}+u\cdot\nabla \Omega_{\theta}=\Omega_r\partial_r u_{\theta}+(\Omega_z-\frac{\Omega_{\theta}}{r})\partial_z u_{\theta}+\frac{1}{r}\Omega_{\theta}u_r-\frac{1}{r}\Omega_ru_{\theta},
$$
therefore
$$
\partial_t \frac{\Omega_{\theta}}{r}
+u\cdot\nabla (\frac{\Omega_{\theta}}{r})=\Omega_r\partial_r (\frac{u_{\theta}}{r} )+(\Omega_z-\frac{\Omega_{\theta}}{r})
\partial_z( \frac{u_{\theta}}{r} ).
$$
Since the component $\Omega_z$ satisfies the following equation
$$
\partial_t \Omega_z+u\cdot\nabla \Omega_z=\Omega_r\partial_ru_z
+(\Omega_z-\frac{\Omega_{\theta}}{r})\partial_z u_z,
$$
then
$$
\partial_t(\Omega_z-\frac{\Omega_{\theta}}{r})
+u\cdot\nabla (\Omega_z-\frac{\Omega_{\theta}}{r})
=\Omega_r\partial_r(u_z-\frac{u_{\theta}}{r})
+(\Omega_z-\frac{\Omega_{\theta}}{r})\partial_z(u_z-\frac{u_{\theta}}{r}).
$$
Thus from the maximum principle and Proposition \ref{nabla}
$$
\begin{aligned}
\big\|(\Omega_z-\frac{\Omega_{\theta}}{r}&)(t)\big\|_{L^{\infty}}
\\&
\lesssim
\int_0^t 
\bigl(\|\Omega_r(\tau)\|_{L^{\infty}}
+\big\|(\Omega_z-\frac{\Omega_{\theta}}{r})(\tau)\big\|_{L^{\infty}}\bigr)
\bigl(||\nabla u||_{L^{\infty}}+||\nabla^2 u||_{L^{\infty}}\bigr)d\tau.
\end{aligned}
$$
Then
$$
\begin{aligned}
\|\Omega_r(t)\|_{L^{\infty}}
&+\big\|(\Omega_z-\frac{\Omega_{\theta}}{r})(t)\big\|_{L^{\infty}}
\\&
\lesssim
\int_0^t 
\bigl(\|\Omega_r(\tau)\|_{L^{\infty}}
+\big\|(\Omega_z-\frac{\Omega_{\theta}}{r})(\tau)\big\|_{L^{\infty}}\bigr)
\bigl(||\nabla u||_{L^{\infty}}+||\nabla^2 u||_{L^{\infty}}\bigr)d\tau.
\end{aligned}
$$
Applying Gronwall inequality gives 
$$
\Omega_r(t)=0\qquad\mbox{and}\qquad
r\Omega_z(t)=\Omega_{\theta}(t),\quad\forall\,t\in\ \RR_+.
$$
Combining the previous estimate with the fact that $u$ is helicoidal, we obtain
\begin{eqnarray}
\Omega\cdot\nabla u&=&\Omega_r\partial_r u+\frac{1}{r} \Omega_{\theta}\partial_{\theta} u+\Omega_z\partial_z u\nonumber\\
&=&\Omega_z
(\partial_{\theta}+\partial_z)(u_re_r+u_\theta e_\theta+u_ze_z)\nonumber\\
&=&\Omega_z(u_re_{\theta}-u_{\theta}e_r).\nonumber
\end{eqnarray}
Which ends the proof of this Proposition.
\end{proof}
An immediate corollary of the above Proposition gives
\begin{corollary}\label{C1}
Let $u$ be an  helicoidal divergence free vector field  solution of the Euler equations, then $\omega=\rot\,u$
verifies 
$$
\left\{\begin{array}{l}
\displaystyle\partial_t \omega_1 +(u_1+yu_3)\partial_x\omega_1 
+(u_2-xu_3)\partial_y\omega_1 =\omega_2u_3-\omega_zu_2,\\
\displaystyle\partial_t \omega_2 +(u_1+yu_3)\partial_x\omega_2 
+(u_2-xu_3)\partial_y\omega_2 =\omega_zu_1-\omega_1u_3,\\
\displaystyle\partial_t \omega_z +(u_1+yu_3)\partial_x\omega_z 
+(u_2-xu_3)\partial_y\omega_z =0,
\end{array}
\right.
$$
with
$$
\partial_x\bigl(u_1+yu_3\bigr)+\partial_y\bigl(u_2-xu_3\bigr)=0.
$$
\end{corollary}
\begin{proof}
By the above Proposition, we have
$$
\omega_1=-y\omega_z\qquad\mbox{and}\qquad\omega_2=x\omega_z
$$
with $\omega_z$ verify 
$$
\partial_t\omega_z+(u_1\partial_x+u_2\partial_y)\omega_z+u_3\partial_z\omega_z=0.
$$
While since
$$
{\mathop{\rm div}}\,\omega =\partial_x\omega_1+\partial_y\omega_2+\partial_z\omega_z=0,
$$
we have
$$
\partial_z\omega_z=-\partial_x\omega_1-\partial_y\omega_2
=y\partial_x\omega_z-x\partial_y\omega_z,
$$
thus
$$
\partial_t \omega_3 +(u_1+yu_3)\partial_x\omega_z +(u_2-xu_3)\partial_y\omega_z =0.
$$
As $\omega_1=-y\omega_z$ and $\omega_2=x\omega_z,$ then
$$
\partial_t \omega_1 +(u_1+yu_3)\partial_x\omega_1 +(u_2-xu_3)\partial_y\omega_1 
=\omega_z(xu_3-u_2)=\omega_2u_3-\omega_zu_2
$$
and
$$
\partial_t \omega_2 +(u_1+yu_3)\partial_x\omega_2 +(u_2-xu_3)\partial_y\omega_2 
=\omega_z(u_1+yu_3)=\omega_zu_1-\omega_1u_3.
$$
Concerning the second point, we have
$$
\begin{aligned}
\partial_x\bigl(u_1+yu_3\bigr)+\partial_y\bigl(u_2-xu_3\bigr)
=\partial_xu_1+\partial_yu_2+y\partial_xu_3-x\partial_yu_3
&=\partial_xu_1+\partial_yu_2-\partial_\theta u_3
\\&
=\partial_xu_1+\partial_yu_2+\partial_zu_3
\\&
=0,
\end{aligned}
$$
and we are done.
\end{proof}
To prove our theorem, we need the following proposition which describes the distribution of the  Fourier coefficients to  transport equation.
\begin{proposition}\label{Fourier}
Under the assumptions in Corollary \ref{C1}. If $\partial_z\omega^0=0$ with 
$\omega^0$ is the initial data, then
$$
\partial_z\omega=0.
$$
\end{proposition}
\begin{proof}
By taking $\partial_z$ to the $\omega_z$ equation, we obtain
$$
\partial_t\partial_z\omega_z+(u_1+yu_3)\partial_x\partial_z\omega_z
+(u_2-xu_3)\partial_y\partial_z\omega_z=
-\partial_z(u_1+yu_3)\partial_x\omega_z
-\partial_z(u_2-xu_3)\partial_y\omega_z.
$$
The fact that ${\mathop{\rm div}}\,u =0,$ $\partial_x(v_1+yv_3)+\partial_y(v_2-xv_3)=0,$
$\omega=(-y\omega_z,x\omega_z,\omega_z)$ and $u_3=yu_1-xu_2,$ leads to
$$
\begin{aligned}
\partial_z(u_1+yu_3)&=x\omega_z+\partial_xu_3-y\partial_xu_1-y\partial_yu_2
\\&
=x\partial_xu_2-x\partial_yu_1+\partial_x(yu_1-xu_2)-y\partial_xu_1-y\partial_yu_2
\\&
=x\partial_xu_2-\partial_y(xu_1)+y\partial_xu_1-\partial_x(xu_2)-y\partial_xu_1-y\partial_yu_2
\\&
=-\partial_y(ru_r).
\end{aligned}
$$
A similar procedure gives rise to
$$
\partial_z(u_2-xu_3)=\partial_x(ru_r).
$$
Hence we obtain
$$
\begin{aligned}
\partial_z(u_1+yu_3)\partial_x\omega_z+\partial_z(u_2-xu_3)\partial_y\omega_z
&=\partial_y(-ru_r)\partial_x\omega_z+\partial_x(ru_r)\partial_y\omega_z
\\&
=\partial_x[\omega_z\partial_y(-ru_r)]+\partial_y[\omega_z\partial_x(ru_r)]
\\&
=\partial_x[-ru_r\partial_y\omega_z]+\partial_y[ru_r\partial_x\omega_z]
\\&
=\partial_x(-ru_r)\partial_y\omega_z+\partial_y(ru_r)\partial_x\omega_z,
\end{aligned}
$$
from which, we infer
$$
\partial_t\partial_z\omega_z+(u_1+yu_3)\partial_x\partial_z\omega_z
+(u_2-xu_3)\partial_y\partial_z\omega_z=0.
$$
Applying maximum principle and Gronwall's inequality, we deduce
$$
\partial_z\omega_z=0,
$$
and as a consequence $\partial_z\omega=0,$ because $\omega=(-y\omega_z,x\omega_z,\omega_z).$ This completes the proof of the proposition.
\end{proof}
To prove Theorem \ref{globale}, we need the following two technical lemmas:
\begin{lemma}\label{LP}
Let $v=(v^1,v^2,v^3)$  be a  divergence free vector field $2\pi$-periodic with respect the third variable, then
$$
\|\nabla_hv\|_{\mathscr{\dot B}^{0}_{\infty,1}}
\lesssim
\|\Omega\|_{\mathscr{\dot B}^{0}_{\infty,1}},
$$
$$
\|\dot \Delta_j\nabla_hv\|_{L^\infty(\RR^2\times]-\pi,\pi[)}
\lesssim
2^{\frac{j}{2}}\|\dot\Delta_j\Omega\|_{L^\infty(\RR^2\times]-\pi,\pi[)},\qquad j\geq0
$$
and
$$
\begin{aligned}
\|\dot S_0v(x_h,z)\|_{L^\infty(\RR^2\times]-\pi,\pi[)}
\lesssim
\|v\|_{L^2(\RR^2\times]-\pi,\pi[)}
+\|\Omega\|_{L^2(\RR^2\times]-\pi,\pi[)},
\end{aligned}
$$
with 
$$
\nabla_h=(\partial_x,\partial_y),\qquad\mbox{and}\qquad\Omega={\mathop{\rm curl}}\,v.
$$
\end{lemma}
\begin{proof}
We have
$$
v(x,y,z)=\sum_{n\in\ZZ}v_n(x,y)e^{inz}
$$
and
$$
\Omega(x,y,z)=\sum_{n\in\ZZ}\Omega_n(x,y)e^{inz},
$$
where $v_n$ is the Fourier coefficients are computed as follows
$$
v_n=\frac{1}{2\pi}\int_{-\pi}^{\pi} v(\cdot,\cdot,z)e^{-inz}dz.
$$
Using
$$
(\Delta_h+\partial^2_z)v=-{\mathop{\rm curl}}\,\Omega,
$$
we find that
$$
(-n^2+\Delta_h)v_{n}^3=\partial_y\Omega_{n}^1-\partial_x\Omega_{n}^2.
$$
Localizes it in horizontal Fourier
$$
\begin{aligned}
\mathcal{F}^h(\dot\Delta_jv_{n}^3)(\xi_h)&=\frac{\xi_1}{n^2+|\xi_h|^2}
\mathcal{F}^h(\dot\Delta_j\Omega_{n}^2)(\xi_h)
-\frac{\xi_2}{n^2+|\xi_h|^2}
\mathcal{F}^h(\dot\Delta_j\Omega_{n}^1)(\xi_h)
\\&
=\frac{\xi_1}{n^2+|\xi_h|^2}\widetilde\varphi(2^{-j}\xi_h)
\mathcal{F}^h(\dot\Delta_j\Omega_{n}^2)(\xi_h)
-\frac{\xi_2}{n^2+|\xi_h|^2}\widetilde\varphi(2^{-j}\xi_h)
\mathcal{F}^h(\dot\Delta_j\Omega_{n}^1)(\xi_h)
\end{aligned}
$$
with $\widetilde\varphi\in\mathcal{S}(\RR^2)$ is a smooth function supported in
$\mathcal{C}= \{\xi\in\RR^2,0<R_1\leq|\xi|\leq R_2\}$  such that
$\widetilde\varphi=1$ on support of $\varphi,$
thus
$$
\mathcal{F}^h(\nabla_h\dot\Delta_jv_{n}^3)(\xi_h)
=\frac{\xi_h\xi_1}{n^2+|\xi_h|^2}\widetilde\varphi(2^{-j}\xi_h)
\mathcal{F}^h(\dot\Delta_j\Omega_{n}^2)(\xi_h)
-\frac{\xi_h\xi_2}{n^2+|\xi_h|^2}\widetilde\varphi(2^{-j}\xi_h)
\mathcal{F}^h(\dot\Delta_j\Omega_{n}^1)(\xi_h).
$$
Let
$$
\mathcal{F}^h(K^i_n)(\xi_h)=
\frac{\xi_h\xi_i}{n^2+|\xi_h|^2}\widetilde\varphi(2^{-j}\xi_h)
\quad\mbox{for}\quad i=1,2,
$$
then
$$
\|K^i_n\|_{L^1}\lesssim\frac{2^{2j}}{n^2+2^{2j}}.
$$
We thus obtain
$$
\|\nabla_h\dot\Delta_jv_{n}^3\|_{L^\infty(\RR^2)}
\lesssim
\bigl(\|K^1_n\|_{L^1}+\|K^2_n\|_{L^1}\bigr)\|\dot\Delta_j\Omega_n\|_{L^\infty(\RR^2)}
\lesssim
\|\dot\Delta_j\Omega_n\|_{L^\infty(\RR^2)}.
$$
Therefore
$$
\|\nabla_hv^3\|_{\mathscr{\dot B}^{0}_{\infty,1}(\RR^2\times]-\pi,\pi[)}
\lesssim
\sum_{n,j}\|\dot\Delta_j\Omega_n\|_{L^\infty(\RR^2)}=
\|\Omega\|_{\mathscr{\dot B}^{0}_{\infty,1}}.
$$
A similar argument gives the same estimate for 
$\|\nabla_hv_i\|_{L^\infty(\RR^2\times]-\pi,\pi[)}$ for $i=1,2.$\\
For the second inequality, we have
\begin{equation}\label{RR}
\sum_{n\in\ZZ}\frac{n^2}{(n^2+\lambda^2)^2}\lesssim{\lambda}^{-1},
\qquad
\sum_{n\in\ZZ}\frac{1}{(n^2+\lambda^2)^2}\lesssim{\lambda}^{-3}
\qquad\forall\;\lambda\geq1
\end{equation}
and
$$
\sum_{n\in\ZZ}|\Omega_n|^2=\|\Omega\|_{L^2(]-\pi,\pi[)}^2.
$$
It follows that for $j\geq0$
$$
\begin{aligned}
\|\dot\nabla_h \Delta_jv\|_{L^\infty}
&\lesssim
\Bigl\{\Big\|\bigl(\sum_{n\in\ZZ}(K^i_n)^2\bigr)^{\frac{1}{2}}\Big\|_{L^1}
+\Big\|\bigl(\sum_{n\in\ZZ}(\kappa_n)^2\bigr)^{\frac{1}{2}}\Big\|_{L^1}\Bigr\}
\Big\|\bigl(\sum_{n\in\ZZ}|\dot \Delta_j\Omega_n|^2\bigr)^{\frac{1}{2}}\Big\|_{L^\infty},
\end{aligned}
$$
with
$$
\mathcal{F}^h\kappa_n(\xi_h)=
\frac{n\xi_h}{n^2+|\xi_h|^2}\widetilde\varphi(2^{-j}\xi_h).
$$
When $|x_h|\geq1,$ we obtained thanks to stationary phase Theorem
$$
|K^i_n(x_h)|+
|\kappa_n(x_h)|
\lesssim
\frac{|n|}{n^2+2^{2j}}\frac{1}{|x_h|^3}
$$
and for $|x_h|\le1,$ we have
$$
|K^i_n(x_h)|+
|\kappa_n(x_h)|
\lesssim
\frac{2^j|n|}{n^2+2^{2j}}+\frac{2^{2j}}{n^2+2^{2j}}.
$$
Finally thanks to \eqref{RR}, we have
$$
\|\dot \Delta_j\nabla_hv\|_{L^\infty(\RR^2\times]-\pi,\pi[)}
\lesssim
2^{\frac{1}{2}j}\|\dot\Delta_j\Omega\|_{\bigl(L^\infty(\RR^2);\,L^2(]-\pi,\pi[)\bigr)}
\lesssim
2^{\frac{1}{2}j}\|\Omega\|_{L^\infty(\RR^2\times]-\pi,\pi[)}.
$$
For the second inequality, we use the fact that
$$
\begin{aligned}
\|\dot S_0v\|_{L^\infty(\RR^2\times]-\pi,\pi[)}
&\lesssim
\sum_{n\in\ZZ,q\le0}\|\dot \Delta_qv_n\|_{L^\infty(\RR^2)}
\\&
\lesssim
\sum_{n\in\ZZ,q\le0}2^q\|\dot \Delta_qv_n\|_{L^2(\RR^2)}
\\&
\lesssim
\sum_{q\le0}2^q\|\dot \Delta_qv_0\|_{L^2(\RR^2)}
+\sum_{n\in\ZZ^*,q\le0}2^q\frac{1}{n}n\|\dot \Delta_qv_n\|_{L^2(\RR^2)}
\\&
\lesssim
\|v\|_{L^2(\RR^2\times]-\pi,\pi[)}
+\sum_{q\le0}2^q\|\partial_z\dot \Delta_qv\|_{L^2(\RR^2\times]-\pi,\pi[)},
\end{aligned}
$$
as
$$
\partial_zv^1=\Omega^2+\partial_xv^3,\quad
\partial_zv^2=-\Omega^1+\partial_yv^3\quad\mbox{and}\quad
\partial_zv^3=-\partial_xv1-\partial_yv^2.
$$
Therefore by virtue of Bernstein's inequality, we obtain
$$
\begin{aligned}
\|\dot S_0v\|_{L^\infty(\RR^2\times]-\pi,\pi[)}
&\lesssim
\|v\|_{L^2(\RR^2\times]-\pi,\pi[)}
+\|\Omega\|_{L^2(\RR^2\times]-\pi,\pi[)}
+\sum_{q\le0}2^q\|\nabla_h\dot \Delta_qv\|_{L^2(\RR^2\times]-\pi,\pi[)}
\\&
\lesssim
\|v\|_{L^2(\RR^2\times]-\pi,\pi[)}
+\|\Omega\|_{L^2(\RR^2\times]-\pi,\pi[)}.
\end{aligned}
$$
This gives the desired result.
\end{proof}
\begin{Rema}\label{LLT}
As 
$$
\partial_zv^1=\Omega^2+\partial_xv^3,\quad
\partial_zv^2=-\Omega^1+\partial_yv^3\quad\mbox{and}\quad
\partial_zv^3=-\partial_xv^1-\partial_yv2,
$$
then
$$
\|\nabla v\|_{\mathscr{\dot B}^{0}_{\infty,1}}
\lesssim
\|\Omega\|_{\mathscr{\dot B}^{0}_{\infty,1}}.
$$
Following a same approach, we obtain
$$
\|\nabla v\|_{\mathscr{\dot B}^{0}_{2,1}}
\lesssim
\|\Omega\|_{\mathscr{\dot B}^{0}_{2,1}}.
$$
\end{Rema}
\begin{lemma}\label{RF}
Let $v$  be divergence free vector field $2\pi$-periodic with respect the third variable, then
$$
\|\nabla_h(xv)\|_{\mathscr{\dot B}^{0}_{\infty,1}}
+\|\nabla_h(yv)\|_{\mathscr{\dot B}^{0}_{\infty,1}}
\lesssim
\|v\|_{L^2(\RR^2\times]-\pi,\pi[)}+
\|(x,y)\Omega\|_{\mathscr{\dot B}^{0}_{\infty,1}}
+\|\Omega\|_{\mathscr{\dot B}^{0}_{\infty,1}},
$$
$$
\|\dot \Delta_j\nabla_h\bigl((x,y)v\bigr)\|_{L^\infty(\RR^2\times]-\pi,\pi[)}
\lesssim
2^{\frac{j}{2}}
\bigl(\|\dot\Delta_j((x,y)\Omega\bigr)\|_{L^\infty(\RR^2\times]-\pi,\pi[)}
+\|\dot\Delta_j\Omega\|_{L^2(\RR^2\times]-\pi,\pi[)}\bigr),\qquad j\geq0
$$
and
$$
\begin{aligned}
\|\dot S_0(x,y)v\|_{L^\infty(\RR^2\times]-\pi,\pi[)}
\lesssim
\|(x,y)v\|_{L^2(\RR^2\times]-\pi,\pi[)}
+\|v\|_{L^2(\RR^2\times]-\pi,\pi[)}
+\|(x,y)\Omega\|_{L^2(\RR^2\times]-\pi,\pi[)},
\end{aligned}
$$
with
$$
\Omega={\mathop{\rm curl}}\,v.
$$
\end{lemma}
\begin{proof}
We have
$$
-\Delta(xv)={\mathop{\rm curl}}\bigl({\mathop{\rm curl}}(xv)\bigr)-\nabla v^1
={\mathop{\rm curl}}(x\Omega)
+\left(\begin{array}{c}\partial_yv^2-\partial_xv^1 \\-\partial_xv^2-\partial_yv^1
 \\-\partial_xv^3-\partial_zv1\end{array}\right)
$$
and $xv$ is $2\pi$-periodic with respect the third variable. Then we deduce forum 
Lemma \ref{LP} that
$$
\|\nabla_h(-\Delta)^{-1}{\mathop{\rm curl}}(x\Omega)\|_{\mathscr{\dot B}^{0}_{\infty,1}}
\lesssim
\|x\Omega\|_{\mathscr{\dot B}^{0}_{\infty,1}}.
$$
For the second terme, we write from the definition of $\Omega$
$$
\partial_zv^1=\Omega^2+\partial_xv^3,
$$
$$
\sum_{n\in\ZZ}\frac{1}{n^2+2^{2j}}\lesssim\left\{\begin{array}{l}
\displaystyle 2^{-j},\quad\mbox{if $j\geq0$}\\
\displaystyle 2^{-2j},\quad\mbox{if $j\leq0$}
\end{array}
\right.
$$
and
$$
\|\dot\Delta_jv_{n}^3\|_{L^\infty(\RR^2)}
\lesssim
\|\dot\Delta_jv^3\|_{L^2(]-\pi,\pi[,\,L^\infty(\RR^2))},
$$
then
$$
\begin{aligned}
\|\nabla_h(-\Delta)^{-1}\partial_xv^3\|_{\mathscr{\dot B}^{0}_{\infty,1}}
&\lesssim
\sum_{j\le0}\|\dot\Delta_jv^3\|_{L^2(]-\pi,\pi[,\,L^\infty(\RR^2))}
+\sum_{j\geq0}\|\dot\Delta_j\nabla_hv^3\|_{L^\infty(\RR^2\times]-\pi,\pi[)}
\\&
\lesssim
\sum_{j\le0}2^j\|\dot\Delta_jv^3\|_{L^2(\RR^2\times]-\pi,\pi[)}
+\sum_{j\geq0}\|\dot\Delta_j\nabla_hv^3\|_{L^\infty(\RR^2\times]-\pi,\pi[)}
\\&
\lesssim
\|v\|_{L^2(\RR^2\times]-\pi,\pi[)}+\|\Omega\|_{\mathscr{\dot B}^{0}_{\infty,1}}
\end{aligned}
$$
and
$$
\begin{aligned}
\|\nabla_h(-\Delta)^{-1}\Omega\|_{\mathscr{\dot B}^{0}_{\infty,1}}
\lesssim
\|\Omega\|_{\mathscr{\dot B}^{0}_{\infty,1}}.
\end{aligned}
$$ 
For the last inequality, we deduce by Lemma \ref{LP} and Bernstein inequality
$$
\|\dot\Delta_j\nabla_h(-\Delta)^{-1}{\mathop{\rm curl}}(x\Omega)\|_{L^\infty}
\lesssim
2^{\frac{1}{2}j}\|\dot\Delta_j(x\Omega)\|_{L^\infty}
$$
and
$$
\|\dot\Delta_j\nabla_h(-\Delta)^{-1}\nabla v\|_{L^\infty}
\lesssim
2^{\frac{1}{2}j}\|\dot\Delta_j\Omega\|_{L^2}
$$
Finally for the last inequality, let us use the fact that
$$
\begin{aligned}
\|\dot S_0(xv)\|_{L^\infty(\RR^2\times]-\pi,\pi[)}
&\lesssim
\sum_{n\in\ZZ,q\le0}\|\dot \Delta_q(xv_n)\|_{L^\infty(\RR^2)}
\\&
\lesssim
\sum_{n\in\ZZ,q\le0}2^q\|\dot \Delta_q(xv_n)\|_{L^2(\RR^2)}
\\&
\lesssim
\sum_{q\le0}2^q\|\dot \Delta_q(xv_0)\|_{L^2(\RR^2)}
+\sum_{n\in\ZZ^*,q\le0}2^q\frac{1}{n}n\|\dot \Delta_q(xv_n)\|_{L^2(\RR^2)}
\\&
\lesssim
\|xv\|_{L^2(\RR^2\times]-\pi,\pi[)}
+\sum_{q\le0}2^q\|\partial_z\dot \Delta_q(xv)\|_{L^2(\RR^2\times]-\pi,\pi[)},
\end{aligned}
$$
as
$$
\begin{aligned}
x\partial_zv^1=x\Omega^2+&\partial_x(xv^3)-v^3,\quad
x\partial_zv^2=-x\Omega^1+\partial_y(xv^3)
\\&
\mbox{and}\quad
x\partial_zu_3=-\partial_x(xv^1)-\partial_y(xv^2)+v^1,
\end{aligned}
$$
Therefore by virtue of Bernstein inequality, we obtain
$$
\begin{aligned}
\|\dot S_0(xv)\|_{L^\infty(\RR^2\times]-\pi,\pi[)}
&\lesssim
\|xv\|_{L^2(\RR^2\times]-\pi,\pi[)}
+\|v\|_{L^2(\RR^2\times]-\pi,\pi[)}
+\|x\Omega\|_{L^2(\RR^2\times]-\pi,\pi[)}
\\&
+\sum_{q\le0}2^q\|\nabla_h\dot \Delta_q(xv)\|_{L^2(\RR^2\times]-\pi,\pi[)}
\\&
\lesssim
\|xv\|_{L^2(\RR^2\times]-\pi,\pi[)}
+\|v\|_{L^2(\RR^2\times]-\pi,\pi[)}
+\|x\Omega\|_{L^2(\RR^2\times]-\pi,\pi[)}.
\end{aligned}
$$
Similar for $yv.$ This achieves the proof of the Lemma.
\end{proof}

\section{Proof of theorem \ref{globale}}
\subsection{Some a priori estimates}
According to \cite{rd}, we deduce the following proposition.
\begin{proposition}\label{henda}
Let $u$ be an helicoidal solution of ${\rm(E)},$ then we have for every $t\in \RR_+$, 
\begin{equation}\label{LI}
\|u(t)\|_{L^{\infty}}+\|\omega(t)\|_{L^{\infty}}
\le
C\bigl(\|u^0\|_{L^2}+\|\omega^0\|_{L^\infty\cap L^2}\bigr)
e^{Ct\|\omega^0_z\|_{L^\infty\cap L^2}}
\end{equation}
and
$$
\|(xu,yu)(t)\|_{L^{\infty}}+\|(x\omega,y\omega)(t)\|_{L^{\infty}}
\le
C\bigl(\|(1,x,y)u^0\|_{L^2}+\|(1,x,y)\omega^0\|_{L^\infty\cap L^2}\bigr)
e^{Ct\|\omega^0_z\|_{L^\infty\cap L^2}}.
$$
\end{proposition}
\begin{proof}
Since $\omega,$ satisfies the following equation
$$
\partial_t \omega+(u\cdot\nabla)\omega=\omega_z(u_re_{\theta}-u_{\theta}e_r),
$$
thus, from the maximum principle we obtain
$$
\|\omega(t)\|_{L^p}
\le
\|\omega^0\|_{L^p}+\int_0^t\|u(\tau)\|_{L^\infty}\|\omega_z(\tau)\|_{L^p}d\tau
\qquad\forall\;p\in[1,\infty].
$$
Since $\omega_z$ satisfies the transport equation, we have
$$
\|\omega_z(t)\|_{L^p}
\le
\|\omega^0_z\|_{L^p},
$$
then
$$
\|\omega(t)\|_{L^p}
\le
\|\omega^0\|_{L^p}+\|\omega^0_z\|_{L^p}\int_0^t\|u(\tau)\|_{L^\infty}d\tau.
$$
To estimate the $L^\infty$ norm of the velocity, we use an argument of P. Serfati \cite{serfati} and Lemma \ref{LP}
$$
\begin{aligned}
\|u(t)\|_{L^\infty(\RR^2\times]-\pi,\pi[)}
&\leq \|\dot S_{0}u\|_{L^\infty}+
\sum_{q\geq  0}\|\dot \Delta_qu\|_{L^\infty}
\\&
\lesssim
\|u\|_{L^2(\RR^2\times]-\pi,\pi[)}+\|\omega\|_{L^2(\RR^2\times]-\pi,\pi[)}
+\sum_{q\geq  0}\|\dot \Delta_qu\|_{L^\infty}.
\end{aligned}
$$  
By Bernstein inequality and Lemma \ref{LP}, we deduce \footnote{ We recall 
the classical fact $\|\dot\Delta_q u\|_{L^p}\approx 2^{-q}\|\dot\Delta_q \omega\|_{L^p}$ uniformly 
in $q,$ for every $p\in [1,+\infty]$.  }
 $$
\sum_{q\geq 0}\|\dot\Delta_q u\|_{L^\infty}
\lesssim 
\|\omega\|_{L^\infty}.  
$$
Consequently, we obtain
$$
\begin{aligned}
\|u(t)\|_{L^\infty(\RR^2\times]-\pi,\pi[)}
&\lesssim
\|u^0\|_{L^2}+\|\omega^0\|_{L^\infty\cap L^2}
+\|\omega^0_z\|_{L^\infty\cap L^2}
\int_0^t\|u(\tau)\|_{L^\infty}d\tau.
\end{aligned}
$$
Using Gronwall's inequality, we have
$$
\|u(t)\|_{L^\infty}
\le
C\bigl\|u^0\|_{L^2}+\|\omega^0\|_{L^\infty\cap L^2}\bigr)
e^{Ct\|\omega^0_z\|_{L^\infty\cap L^2}}.
$$
By maximum principle, Gronwall's inequality and inequality \eqref{LI}, we deduce
$$
\|(x,y)\omega\|_{L^\infty_t(L^p)}
\lesssim
\|(x,y)\omega^0\|_{L^p}+t\|\omega\|_{L^\infty_t(L^p)}\|u\|_{L^\infty_t(L^\infty)}
\qquad\forall\;p\in[1,\infty].
$$
For concluded the proof stays to controlled $\|(x,y)u\|_{L^2},$ we have
$$
\partial_t(xu)+(u\cdot\nabla)(xu)+\nabla(xp)=\left(\begin{array}{c}p+(u_1)^2 \\ u_1u_2\\u_1u_3\end{array}\right),
$$
from which, we deduce
$$
\begin{aligned}
\frac{1}{2}\frac{d}{dt}\|xu\|_{L^2}^2
&=\int_{\RR^2\times]-\pi,\pi[} xu_1\bigl(2p+(u_1)^2+(u_2)^2+(u_3)^2\bigr)dxdydz
\\&
\lesssim
\|xu\|_{L^2}\bigl(\|p\|_{L^2}+\|u\|_{L^\infty}\|u\|_{L^2}\bigr).
\end{aligned}
$$
As
$$
-\Delta p={\mathop{\rm div}}{\mathop{\rm div}}(u\otimes u),
$$
Parseval's equality and the following inequality 
$$
\frac{2^{2j}+n^2+|n|2^j}{n^2+2^{2j}}\lesssim1
$$
we follow the same approach in the proof of Lemma \ref{LP}, we obtain
$$
\|p\|_{L^2(\RR^2\times]-\pi,\pi[)}
\lesssim
\|u\otimes u\|_{L^2}
\lesssim
\|u\|_{L^\infty}\|u\|_{L^2}.
$$
Then
\begin{equation}\label{VS}
\|xu\|_{L^2}
\lesssim
\|xu^0\|_{L^2}+t\|u\|_{L^\infty_t(L^2)}\|u\|_{L^\infty_t(L^\infty)}
\lesssim
\|xu^0\|_{L^2}+t\|u^0\|_{L^2}\|u\|_{L^\infty_t(L^\infty)}
\le
C_0e^{C_0t}.
\end{equation}
As a consequence, we obtain
$$
\begin{aligned}
\|xu\|_{L^\infty(\RR^2\times]-\pi,\pi[)}
&\leq \|\dot S_{0}(xu)\|_{L^\infty}+
\sum_{q\geq  0}\|\dot \Delta_q(xu)\|_{L^\infty}
\\&
\lesssim
\|u\|_{L^2}+\|xu\|_{L^2}+\|\omega\|_{L^2}
+\|x\omega\|_{L^2}
\\&
\lesssim
\|u^0\|_{L^2}+\|xu^0\|_{L^2}
+\bigl(\|u^0\|_{L^2}+\|\omega^0\|_{L^\infty\cap L^2}+\|x\omega^0\|_{L^2}\bigr)
e^{Ct\|\omega^0_z\|_{L^\infty\cap L^2}}
\\&
\le
C\bigl(\|u^0\|_{L^2}+\|xu^0\|_{L^2}+\|\omega^0\|_{L^\infty\cap L^2}+\|x\omega^0\|_{L^2}\bigr)
e^{Ct\|\omega^0_z\|_{L^\infty\cap L^2}}.
\end{aligned}
$$  
And a similar argument gives the same estimate for $\|yu\|_{L^\infty}.$ Hence the proposition.
\end{proof}
The evolution of the quantity $\|\nabla u\|_{L^1_t(L^\infty)}$ is related to the following result: 
\begin{proposition}\label{lipsh}
There exists a decomposition $(\tilde\omega_{q,n})_{(q,n)\in \ZZ^2}$ of the vorticity $\omega$ such that\\
i) For every $t\in \RR_+,$ we have
$$
 \omega=\displaystyle {\sum_{(q,n)\in \ZZ^2}}\tilde \omega_{q,n}e^{inz}
$$
and
$$
\dv\, \tilde \omega _{q,n}(t,x)=0.
$$
ii) For every $(q,n)\in \ZZ^2,$ we have
$$
||\tilde \omega_{q,n}(t)||_{L^{\infty}}
\lesssim 
\bigl(||\dot \Delta _q \omega^0_n||_{L^{\infty}}
+c_q\|\omega^0_{z,n}\|_{\dot B^0_{2,1}}\bigr)e^{C_0t}
$$
where $C_0$ is a constant depending on $u^0$ and $c_q\in\ell^1(\ZZ)$ 
(see Proposition \ref{proposition}).\\
iii) For every $(j,q,n)\in \ZZ^3,$ we have
$$
||\dot \Delta_j\tilde \omega_{q,n}(t)||_{L^{\infty}}
\leq 
C_02^{-|j-q|}e^{C_0U(t)}\bigl(||\dot \Delta_q \omega^0_n||_{L^{\infty}}
+c_q\|\omega^0_{z,n}\|_{\dot B^0_{2,1}}\bigr),
$$
with $U(t):=||\tilde u||_{L_t^1(B_{\infty,1}^1)}+||u||_{L_t^1(\dot B_{2,1}^1)}.$ \end{proposition}
\begin{proof}
We will localize in vertical frequency the initial data and denote by $\tilde \omega_{n}$   the unique global vector-valued solution of the problem 
$$
\left\{\begin{array}{l}
\displaystyle\partial_t \tilde\omega_{1,n} +(\tilde u\cdot\nabla_h)\tilde\omega_{1,n} 
 =-\tilde\omega_{z,n}\tilde u_2,\\
\displaystyle\partial_t \tilde\omega_{2,n} +(\tilde u\cdot\nabla_h)\tilde\omega_{2,n} 
 =\tilde\omega_{z,n}\tilde u_1,\\
\displaystyle\partial_t\tilde\omega_{z,n} +(\tilde u\cdot\nabla_h)\tilde\omega_{z,n} =0,\\
\tilde \omega_{n|t=0}=\tilde \omega_{n}(0)
\end{array} 
\right.
$$
where 
$$
\tilde \omega_{n}(0)=\begin{pmatrix}-y\omega^0_{n,z}\\ 
x\omega^0_{n,z}\\ 
\omega^0_{n,z}\\ 
\end{pmatrix},
\qquad
\widetilde{u}=\left(\begin{array}{c}u_1+yu_3 \\u_2-xu_3\end{array}\right),
\qquad
\nabla_h=(\partial_x,\partial_y)
$$
and
$$
\omega^0_{n,z}
=\frac{1}{2\pi}\int_{-\pi}^{\pi}\omega^0_{z}e^{-inz}dz
=\partial_xu^0_{n,2}-\partial_yu^0_{n,1}.
$$
By Proposition \ref{Fourier}, we deduce that $\tilde\omega_{n}$ is the Fourier coefficients of $\omega,$ i.e, $\tilde\omega_{n}=\omega_{n}.$ Thus
$$
\omega=\sum_{n\in\ZZ}\omega_{n}e^{inz}
\quad\mbox{and}\quad
\|\omega\|_{\mathscr{\dot B}^s_{p,r}}
=
\sum_{n\in\ZZ}\|\omega _{n}\|_{\dot B^s_{p,r}}.
$$
We will use for this purpose a new approach similar to  \cite{hmidi}, which consists to  linearize properly the Fourier  of transport equation. For that, we will localize in frequency the initial data and denote by $\tilde \omega_q$   the unique global vector-valued solution of the problem 
$$
\left\{\begin{array}{l}
\displaystyle\partial_t \tilde\omega_{1,q,n} +(\tilde u\cdot\nabla_h)\tilde\omega_{1,q,n} 
 =-\tilde\omega_{z,q,n}\tilde u_2,\\
\displaystyle\partial_t \tilde\omega_{2,q,n} +(\tilde u\cdot\nabla_h)\tilde\omega_{2,q,n} 
 =\tilde\omega_{z,q,n}\tilde u_1,\\
\displaystyle\partial_t\tilde\omega_{z,q,n} +(\tilde u\cdot\nabla_h)\tilde\omega_{z,q,n} =0,\\
\tilde \omega_{q,n|t=0}=\tilde \omega_{q,n}(0)
\end{array} 
\right.
$$
where 
$$
\tilde \omega_{q,n}(0)=\begin{pmatrix}-y\dot \Delta_q \omega^0_{n,z}\\ 
x\dot \Delta_q \omega^0_{n,z}\\ 
\dot \Delta_q \omega^0_{n,z}\\ 
\end{pmatrix},
\qquad
\widetilde{u}=\left(\begin{array}{c}u_1+yu_3 \\u_2-xu_3\end{array}\right),
\qquad
\nabla_h=(\partial_x,\partial_y)
$$
and
$$
\dot \Delta_q \omega^0_{n,z}
=\frac{1}{2\pi}\int_{-\pi}^{\pi}\dot \Delta_q \omega^0_{z}e^{-inz}dz
=\partial_x\dot\Delta_qu^0_{n,2}-\partial_y\dot\Delta_qu^0_{n,1}.
$$
In addition by linearity and uniqueness
\begin{equation*}
\omega=\sum_{(q,n)\in\ZZ^2}\tilde\omega _{q,n}e^{inz}.
\end{equation*}
Since $
\dv\,\tilde\omega_{q,n}(0)=-y\partial_x(\dot\Delta_q \omega^0_{z,n})+
x\partial_y(\dot\Delta_q \omega^0_{z,n})
=\partial_\theta(\dot\Delta_q\omega^0_{z,n})=0$ and $\tilde\omega_q(0)=r\dot \Delta_q \omega^0_ze_{\theta}
+\dot \Delta_q \omega^0_z e_z,$ then applying Proposition \ref{mahdi} gives $\tilde \omega_{q,n}=r\tilde\omega_{q,z,n} e_{\theta}+\tilde \omega_{q,z,n}e_z$ and
\begin{equation}
\left\{ \begin{array}{lll}
\partial_t \tilde \omega_{q,n}+(u.\nabla)\tilde \omega_{q,n}=
\tilde \omega_{q,n,z}(u_re_{\theta}-u_{\theta}e_r)\\
\tilde \omega_{q,n|t=0}=\tilde \omega_{q,n}(0).
\end{array} \right.
\end{equation}
Applying the maximum principle and using Propositions \ref{henda}, \ref{proposition}, we obtain
\begin{equation}\label{KL}
\begin{aligned}
||\tilde \omega_{q,n}(t)||_{L^{\infty}}
&\lesssim 
\|\tilde \omega_{q,n}(0)\|_{L^\infty}+t\|\tilde \omega_{q,n,z}\|_{L^\infty_t(L^\infty)}
\|u\|_{L^\infty_t(L^\infty)}
\\&
\lesssim
\bigl(||\dot \Delta _q \omega^0_n||_{L^{\infty}}+c_q\|\omega^0_{z,n}\|_{\dot B^0_{2,1}}\bigr)e^{C_0t}
\\&
\lesssim
2^q\bigl(||\dot \Delta _q \omega^0_n||_{L^2}+c_q\|\omega^0_{z,n}\|_{\dot B^0_{1,1}}\bigr)e^{C_0t}.
\end{aligned}
\end{equation}
This complete the proof of i)-ii) of the proposition.\\
Let us now move to the proof of iii) which is the main property of the above decomposition.
Remark first that the desired estimate is equivalent to
\begin{equation}\label{4.5}
||\dot \Delta_j \tilde \omega_{q,n}(t)||_{L^{\infty}}
\leq 
C2^{j-q}e^{CU(t)}\bigl(||\dot \Delta_q \omega^0_n||_{L^{\infty}}
+c_q\|\omega^0_{z,n}\|_{\dot B^0_{2,1}}\bigr)
\end{equation}
and
\begin{equation}\label{4.6}
||\dot \Delta_j \tilde \omega_{q,n}(t)||_{L^{\infty}}
\leq 
C2^{q-j}e^{CU(t)}\bigl(||\dot \Delta_q \omega^0_n||_{L^{\infty}}
+c_q\|\omega^0_{z,n}\|_{\dot B^0_{2,1}}\bigr),
\end{equation}
with $c_q\in\ell^1(\ZZ).$
From Corollary \ref{C1} , it is plain that the $\tilde{\omega}_{q,n}$ is solution of
\begin{equation}\label{VH}
\left\{\begin{array}{l}
\displaystyle\partial_t \tilde\omega_{1,q,n} +(\tilde u\cdot\nabla_h)\tilde\omega_{1,q,n} 
 =\tilde\omega_{2,q,n}u_3-\tilde\omega_{z,q,n}u_2,\\
\displaystyle\partial_t \tilde\omega_{2,q,n} +(\tilde u\cdot\nabla_h)\tilde\omega_{2,q,n} 
 =\tilde\omega_{z,q,n} u_1-\tilde\omega_{1,q,n}u_3,\\
\displaystyle\partial_t\tilde\omega_{z,q,n} +(\tilde u\cdot\nabla_h)\tilde\omega_{z,q,n} =0,
\end{array}
\right.
\end{equation}
with
$$
\tilde u=(u_1+yu_3,u_2-xu_3)\quad\mbox{and}\quad
\partial_x\bigl(u_1+yu_3\bigr)+\partial_y\bigl(u_2-xu_3\bigr)=0.
$$
\subsubsection*{Step 1: \, Proof of (\ref{4.5}) }\  
Applying Corollary \ref{saoussen} to \eqref{VH}
\begin{equation}\label{4.7}
e^{-C\|\tilde u\|_{L^1_t(\dot B^1_{\infty,1})}}
||\tilde \omega_{q,n}(t)||_{\dot B_{\infty,\infty}^{-1}}
\lesssim 
||\tilde\omega_{q,n}(0)||_{\dot B_{\infty,\infty}^{-1}}
+\int_0^t e^{-C\|\tilde u\|_{L^1_{\tau}(\dot B^1_{\infty,1})}}
||\tilde \omega_{i,q,n} u_j(\tau)||_{\dot B_{\infty,\infty}^{-1}}d\tau.
\end{equation}
To estimate the integral term we write in view of Bony's decomposition
\begin{eqnarray}
||\tilde \omega_{i,q,n}u_j||_{\dot B_{\infty,\infty}^{-1}}
&\leq& 
||T_{\tilde \omega_{i,q,n}} u_j||_{\dot B_{\infty,\infty}^{-1}}
+||T_{u_j}\tilde \omega_{i,q,n}||_{\dot B_{\infty,\infty}^{-1}}
+||R(\tilde \omega_{i,q,n},u_j)||_{\dot B_{\infty,\infty}^{-1}}
\nonumber\\
&\lesssim& 
||u||_{L^{\infty}}||\tilde \omega_{z,q,n}||_{\dot B_{\infty,\infty}^{-1}}
+||R(\tilde \omega_{i,q,n},u_j)||_{\dot B_{\infty,\infty}^{-1}}.
\nonumber
\end{eqnarray}
The remainder term can be treated as follows
\begin{eqnarray}
||R(\tilde \omega_{i,q,n},u_j)||_{\dot B_{\infty,\infty}^{-1}}
&\lesssim&
\sup_k\sum_{\ell\geq k-3}||\dot\Delta_{\ell} \tilde \omega_{i,q,n}||_{L^{\infty}}
||\tilde{\dot\Delta}_{\ell}u_j||_{L^2}\nonumber\\
&\lesssim& ||\tilde \omega_{q,n}||_{\dot B_{\infty,\infty}^{-1}}||u||_{\dot B_{2,1}^{1}}.\nonumber
\end{eqnarray}
It follows that
$$
||\tilde \omega_{i,q,n}u_j||_{\dot B_{\infty,\infty}^{-1}}
\lesssim 
||u||_{\dot B_{2,1}^{1}}||\tilde \omega_{q,n}||_{\dot B_{\infty,\infty}^{-1}}.
$$
Inserting this estimate into (\ref{4.7}) we get
$$
e^{-C\|\tilde u\|_{L^1_t(\dot B^1_{\infty,1})}}
||\tilde \omega_{q,n}(t)||_{\dot B_{\infty,\infty}^{-1}}
\lesssim 
||\tilde\omega_{q,n}(0)||_{\dot B_{\infty,\infty}^{-1}}
+\int_0^t||u(\tau)||_{\dot B_{2,\infty}^{1}}
e^{-C\|\tilde u\|_{L^1_{\tau}(\dot B^1_{\infty,1})}}
||\tilde \omega_{q,n}(\tau)||_{\dot B_{\infty,\infty}^{-1}}d\tau.
$$
Hence we obtain by Gronwall's inequality and unsung Proposition \ref{proposition}
\begin{equation}\label{LT}
\begin{aligned}
||\tilde \omega_{q,n}(t)||_{\dot B_{\infty,1}^{-1}}
&\leq
C(||\dot\Delta_q \omega^0_n||_{\dot B_{\infty,\infty}^{-1}}
+||\tilde{h}^1_q\ast\omega^0_{z,n}||_{B_{\infty,\infty}^{-1}}
+||\tilde{h}^2_q\ast\omega^0_{z,n}||_{B_{\infty,\infty}^{-1}})
\\&
\times
e^{C\|\tilde u\|_{L^1_{t}(\dot B^1_{\infty,1})}+C\|u\|_{L^1_t(\dot B^1_{2,1})}}
\\& 
\le
C2^{-q}\bigl(||\dot \Delta_q \omega^0_n||_{L^{\infty}}
+c_q\|\omega^0_{z,n}\|_{\dot B^0_{2,1}}\bigr)e^{CU(t)}.
\end{aligned}
\end{equation}
This gives by definition
$$
||\dot\Delta_j \tilde \omega_q(t)||_{L^{\infty}}
\leq 
C2^{j-q}\bigl(||\dot \Delta_q \omega^0_n||_{L^{\infty}}
+c_q\|\omega^0_{z,n}\|_{\dot B^0_{2,1}}\bigr)e^{CU(t)}.
$$
\subsubsection*{Step 2: \, Proof of (\ref{4.6}) }\ 
 The solution $\tilde \omega_q$ has three components in the cartesian basis $\tilde \omega_{q,n}=(\tilde \omega_{1,q,n},\tilde \omega_{2,q,n},\tilde \omega_{z,q,n}).$
It's clear that $\tilde \omega_{1,q,n}$ is solution of
$$
\left\{ \begin{array}{lll}
\partial_t \tilde \omega_{1,q,n}+(\tilde u\cdot\nabla_h)\tilde \omega_{1,q,n}
=-\tilde \omega_{z,q,n}\tilde u_2\\
\tilde \omega_{1,q,n|t=0}=\tilde\omega_{1,q,n}(0).
\end{array}\right.
$$
Then, we obtain from Corollary \ref{saoussen}
$$
e^{-C\|\tilde u\|_{L^1_t(\dot B^1_{\infty,1})}}||\tilde\omega_{1,q,n}(t)||_{\dot B_{\infty,1}^1}
\lesssim 
||\tilde \omega_{1,q,n}(0)||_{\dot B_{\infty,1}^1}+\int_0^t 
e^{-C\|\tilde u\|_{L^1_{\tau}(\dot B^1_{\infty,1})}}
||\tilde\omega_{z,q,n}\tilde u_2||_{\dot B_{\infty,1}^1}d\tau.
$$
From Bony's decomposition, we get 
$$
\begin{aligned}
e^{-C\|\tilde u\|_{L^1_t(\dot B^1_{\infty,1})}}||\tilde\omega_{1,q,n}(t)||_{\dot B_{\infty,1}^1}
\lesssim 
||\tilde \omega_{1,q,n}(0)||_{\dot B_{\infty,1}^1}
&+\int_0^t  e^{-C\|\tilde u\|_{L^1_{\tau}(\dot B^1_{\infty,1})}}
||\tilde\omega_{z,q,n}||_{\dot B_{\infty,1}^1}||\tilde u||_{L^{\infty}}d\tau
\\&+\int_0^t  e^{-C\|\tilde u\|_{L^1_{\tau}(\dot B^1_{\infty,1})}}||\tilde u||_{\dot B_{\infty,1}^1}
||\tilde\omega_{z,q,n}||_{L^{\infty}}d\tau.
\end{aligned}
$$ 
The analysis will be exactly the same for $\tilde \omega_{2,q,n}$,  because it satisfies the following equation
$$
\left\{ \begin{array}{lll}
\partial_t \tilde \omega_{2,q,n}+(\tilde u\cdot\nabla_h)\tilde \omega_{2,q,n}=
\tilde\omega_{z,q,n}\tilde u_1\\
\tilde \omega_{2,q,n|t=0}=\tilde \omega_{2,q,n}(0).
\end{array}\right.
$$
Since  $\tilde \omega_{z,q,n}$ satisfies
$$
\left\{ \begin{array}{lll}
\partial_t \tilde \omega_{z,q,n}+(\tilde u\cdot\nabla_h)\tilde \omega_{z,q,n}=0\\
\tilde \omega_{z,q,n|t=0}=\dot\Delta_q \omega_{z,n}^0,
\end{array}\right.
$$
then
$$
e^{-C\|\tilde u\|_{L^1_t(\dot B^1_{\infty,1})}}||\tilde \omega_{z,q,n}(t)||_{\dot B_{\infty,1}^1}
\lesssim 
||\dot\Delta_q \omega_{z,n}^0||_{\dot B_{\infty,1}^1}
$$
and
$$
||\tilde \omega_{z,q,n}(t)||_{L^\infty}
\le
||\dot\Delta_q \omega_{z,n}^0||_{L^\infty}
\lesssim
2^q||\dot\Delta_q \omega_{z,n}^0||_{L^2}.
$$
Finally we obtain
$$
\begin{aligned}
e^{-C\|\tilde u\|_{L^1_t(\dot B^1_{\infty,1})}}||\tilde \omega_{q,n}(t)||_{\dot B_{\infty,1}^1}
&\lesssim 
||\dot\Delta_q\omega^0_n||_{\dot B_{\infty,1}^1}
+||\tilde{h}^1_q\ast\omega^0_{z,n}||_{\dot B_{\infty,1}^1}
+||\tilde{h}^2_q\ast\omega^0_{z,n}||_{\dot B_{\infty,1}^1}
\\&
+\int_0^t e^{-C\|\tilde u\|_{L^1_{\tau}(\dot B^1_{\infty,1})}}
||\tilde\omega_{z,q,n}||_{\dot B_{\infty,1}^1}
||\tilde u||_{L^\infty}d\tau
\\&
+\int_0^t  e^{-C\|\tilde u\|_{L^1_{\tau}(\dot B^1_{\infty,1})}}
||\tilde u||_{\dot B_{\infty,1}^1}||\tilde\omega_{z,q,n}||_{L^{\infty}}d\tau.
\end{aligned}
$$
So according to Gronwall's inequality and using Propositions \ref{henda} and \ref{proposition} (see Appendix), we obtain
\begin{equation}\label{TL}
||\tilde \omega_{q,n} (t)||_{\dot B_{\infty,1}^1}
\leq 
C2^q \bigl(||\dot \Delta_q \omega^0||_{L^{\infty}}
+c_q\|\omega^0_{z,n}\|_{\dot B^0_{2,1}}\bigr)e^{CU(t)}.
\end{equation}
This can be written
$$
||\dot \Delta _j\tilde \omega_{q,n}(t)||_{L^{\infty}}
\leq C2^{q-j}\bigl(||\dot \Delta_q \omega^0||_{L^{\infty}}
+c_q\|\omega^0_{z,n}\|_{\dot B^0_{2,1}}\bigr)e^{CU(t)}.
$$
Hence,  the desired result.
\end{proof}
So that $x\tilde\omega_{q,n}$ and $y\tilde\omega_{q,n}$ satisfies
$$
\begin{aligned}
\partial_t(x\tilde\omega_{q,n})+(\tilde u\cdot\nabla_h)(x\tilde\omega_{q,n})
&=\tilde u_1\tilde\omega_{q,n}
+\left(\begin{array}{c}-\tilde u_2\tilde\omega_{2,q,n} \\ \tilde u_1\tilde\omega_{2,q,n} 
\\0\end{array}\right)
\end{aligned}
$$
and
$$
\begin{aligned}
\partial_t(y\tilde\omega_{q,n})+(\tilde u\cdot\nabla_h)(y\tilde\omega_{q,n})
&=\tilde u_2\tilde\omega_{q,n}
+\left(\begin{array}{c}\tilde u_2\tilde\omega_{1,q,n} \\ -\tilde u_1\tilde\omega_{1,q,n}
 \\0\end{array}\right).
 \end{aligned}
$$
We follow the same proof of the previous proposition and using 
Corollary \ref{CM}, we obtain the following proposition. 
\begin{proposition}\label{PC}
There exists $C_0$ is a constant depending on $u^0$ and $c_q\in\ell^1(\ZZ),$ such that.\\
i) For every $(q,n)\in \ZZ^2,$ we have
$$
||(x,y)\tilde \omega_{q,n}(t)||_{L^{\infty}}
\le
C_0\bigl(||\dot \Delta _q\{(1,x,y)\omega^0_n\}||_{L^{\infty}}
+c_q\|\omega^0_{n}\|_{\dot B^0_{2,1}}
+c_q\|\omega^0_{z,n}\|_{\dot B^0_{1,1}}\bigr)e^{C_0t}.
$$ 
ii) For every $(j,q,n)\in \ZZ^3,$ we have
$$
\begin{aligned}
||\dot \Delta_j\{(x,y)\tilde \omega_{q,n}\}(t)||_{L^{\infty}}
&\leq 
C_02^{-|j-q|}e^{C_0U(t)}\exp(C_0e^{C_0t})
\\&
\times
\Bigl(||\dot \Delta_q\{(1,x,y)\omega^0_n\}||_{L^{\infty}}+||\dot \Delta _q \omega^0_n||_{L^2}
+c_q\|\omega^0_{n}\|_{\dot B^0_{2,1}}+c_q\|\omega^0_{z,n}\|_{\dot B^0_{1,1}}
\Bigr).
\end{aligned}
$$
with $U(t):=||\tilde u||_{L_t^1(\dot B_{\infty,1}^1)}+||u||_{L_t^1(\dot B_{2,1}^1)}.$
\end{proposition}
\begin{proof}
i) According to  maximum principle, Propositions  \ref{henda}, \ref{lipsh}, \ref{proposition} and Corollary \ref{CM}, we have 
$$
\begin{aligned}
||(x,y)\tilde \omega_{q,n}(t)||_{L^{\infty}}
&\lesssim
\|(x,y)\dot\Delta_q\omega^0_{z,n}\|_{L^\infty}+
\|xy\dot\Delta_q\omega^0_{z,n}\|_{L^\infty}+
\|x^2\dot\Delta_q\omega^0_{z,n}\|_{L^\infty}+
\|y^2\dot\Delta_q\omega^0_{z,n}\|_{L^\infty}
\\&
+\int_0^t\|\tilde u\|_{L^\infty}\|\tilde\omega_{q,n}\|_{L^\infty}d\tau
\\&
\lesssim 
\bigl(||\dot \Delta _q\{(1,x,y)\omega^0_n\}||_{L^{\infty}}
+c_q\|\omega^0_{n}\|_{\dot B^0_{2,1}}
+c_q\|\omega^0_{z,n}\|_{\dot B^0_{1,1}}\bigr)e^{C_0t}.
\end{aligned}
$$ 
ii) Corollary \ref{saoussen}, implies that
$$
\begin{aligned}
e^{-C\|\tilde u\|_{L^1_t(\dot B^1_{\infty,1})}}
&||(x,y)\tilde \omega_{q,n}(t)||_{\dot B^1_{\infty,1}}
\lesssim
\|(x,y)\dot\Delta_q\omega^0_{z,n}\|_{\dot B^1_{\infty,1}}+
\|xy\dot\Delta_q\omega^0_{z,n}\|_{\dot B^1_{\infty,1}}+
\|x^2\dot\Delta_q\omega^0_{z,n}\|_{\dot B^1_{\infty,1}}
\\&
+\|y^2\dot\Delta_q\omega^0_{z,n}\|_{\dot B^1_{\infty,1}}
+\int_0^te^{-C\|\tilde u\|_{L^1_{\tau}(\dot B^1_{\infty,1})}}\|\tilde u\|_{L^\infty}\|\tilde\omega_{q,n}\|_{\dot B^1_{\infty,1}}d\tau
\\&
+\int_0^te^{-C\|\tilde u\|_{L^1_{\tau}(\dot B^1_{\infty,1})}}\|\tilde\omega_{q,n}\|_{L^\infty}\|\tilde u\|_{\dot B^1_{\infty,1}}d\tau,
\end{aligned}
$$ 
this along with Propositions  \ref{henda}, \ref{lipsh}, \ref{proposition}, Corollary \ref{CM}
and inequalities \eqref{KL}, \eqref{TL}, ensures that
$$
\begin{aligned}
e^{-C\|\tilde u\|_{L^1_{t}(\dot B^1_{\infty,1})}}||(x,y)\tilde \omega_{q,n}(t)||_{\dot B^1_{\infty,1}}
&\lesssim 
2^q\bigl(||\dot \Delta _q\{(1,x,y)\omega^0_n\}||_{L^{\infty}}
+c_q\|\omega^0_{n}\|_{\dot B^0_{2,1}}
+c_q\|\omega^0_{z,n}\|_{\dot B^0_{1,1}}\bigr)
\\&
+2^q \bigl(||\dot \Delta_q \omega^0_n||_{L^{\infty}}
+c_q\|\omega^0_{z,n}\|_{\dot B^0_{2,1}}\bigr)e^{C_0t+C\|u\|_{L^1_t(\dot B^1_{2,1})}}
\\&
+2^q\bigl(||\dot \Delta _q \omega^0_n||_{L^2}+c_q\|\omega^0_{n}\|_{\dot B^0_{2,1}}
+c_q\|\omega^0_{n}\|_{\dot B^0_{1,1}}\bigr)e^{C_0t}.
\end{aligned}
$$ 
As a consequence, we obtain
$$
\begin{aligned}
||\dot \Delta_j\{(x,y)\tilde \omega_{q,n}\}(t)||_{L^{\infty}}
&\leq 
C_02^{q-j}e^{C_0(t+U(t))}
\\&
\times
\Bigl(||\dot \Delta_q\{(1,x,y)\omega^0_n\}||_{L^{\infty}}
+c_q\|\omega^0_{n}\|_{\dot B^0_{2,1}}+c_q\|\omega^0_{z,n}\|_{\dot B^0_{1,1}}
+||\dot \Delta _q \omega^0_n||_{L^2}\Bigr).
\end{aligned}
$$
To prove the  estimate
$$
\begin{aligned}
||\dot \Delta_j\{(x,y)\tilde \omega_{q,n}\}(t)||_{L^{\infty}}
&\leq 
C_02^{j-q}e^{C_0(t+U(t))}
\\&
\times
\Bigl(||\dot \Delta_q\{(1,x,y)\omega^0_n\}||_{L^{\infty}}
+c_q\|\omega^0_{z,n}\|_{\dot B^0_{2,1}}+c_q\|\omega^0_{z,n}\|_{\dot B^0_{1,1}}
+||\dot \Delta _q \omega^0_n||_{L^2}\Bigr),
\end{aligned}
$$
we use the fact that
$$
\begin{aligned}
\partial_t(x\tilde\omega_{q,n})+(\tilde u\cdot\nabla_h)(x\tilde\omega_{q,n})
=u_1\tilde\omega_{q,n}+u_3(y\tilde\omega_{q,n})+
\left(\begin{array}{c}-u_2\tilde\omega_{2,q,n}+u_3(x\tilde\omega_{2,q,n}) \\ 
u_1\tilde\omega_{2,q,n}+u_3(y\tilde\omega_{2,q,n}) 
\\0\end{array}\right)
\end{aligned}
$$
and
$$
\begin{aligned}
\partial_t(y\tilde\omega_{q,n})+(\tilde u\cdot\nabla_h)(y\tilde\omega_{q,n})
 =
 u_2\tilde\omega_{q,n}-u_3(x\tilde\omega_{q,n})
+\left(\begin{array}{c}u_2\tilde\omega_{1,q,n}-u_3(x\tilde\omega_{1,q,n}) 
\\ -u_1\tilde\omega_{1,q,n}-u_3(y\tilde\omega_{1,q,n})
 \\0
 \end{array}\right).
 \end{aligned}
$$
This along with Corollary \ref{saoussen} leads to
$$
\begin{aligned}
e^{-C\|\tilde u\|_{L^1_t(\dot B^1_{\infty,1})}}
&||(x,y)\tilde \omega_{q,n}(t)||_{\dot B^{-1}_{\infty,\infty}}
\lesssim
\|(x,y)\dot\Delta_q\omega^0_{z,n}\|_{\dot B^{-1}_{\infty,\infty}}+
\|xy\dot\Delta_q\omega^0_{z,n}\|_{\dot B^{-1}_{\infty,\infty}}+
\|x^2\dot\Delta_q\omega^0_{z,n}\|_{\dot B^{-1}_{\infty,\infty}}
\\&
+\|y^2\dot\Delta_q\omega^0_{z,n}\|_{\dot B^{-1}_{\infty,\infty}}
+\int_0^te^{-C\|\tilde u\|_{L^1_{\tau}(\dot B^1_{\infty,1})}}\|u\|_{L^\infty}
\|(1,x,y)\tilde\omega_{q,n}\|_{\dot B^{-1}_{\infty,\infty}}d\tau
\\&
+\int_0^te^{-C\|\tilde u\|_{L^1_{\tau}(\dot B^1_{\infty,1})}}
\|(1,x,y)\tilde\omega_{q,n}\|_{\dot B^{-1}_{\infty,\infty}}\|u\|_{\dot B^{1}_{2,1}}d\tau.
\end{aligned}
$$ 
Applying Gronwall's inequality, yields
$$
\begin{aligned}
e^{-C\|\tilde u\|_{L^1_t(\dot B^1_{\infty,1})}}
&||(x,y)\tilde \omega_{q,n}(t)||_{\dot B^{-1}_{\infty,\infty}}
\lesssim
\Bigl\{\|(x,y)\dot\Delta_q\omega^0_{z,n}\|_{\dot B^{-1}_{\infty,\infty}}+
\|xy\dot\Delta_q\omega^0_{z,n}\|_{\dot B^{-1}_{\infty,\infty}}+
\|x^2\dot\Delta_q\omega^0_{z,n}\|_{\dot B^{-1}_{\infty,\infty}}
\\&
+\|y^2\dot\Delta_q\omega^0_{z,n}\|_{\dot B^{-1}_{\infty,\infty}}
+\int_0^te^{-C\|\tilde u\|_{L^1_{\tau}(\dot B^1_{\infty,1})}}\|u\|_{L^\infty}
\|\tilde\omega_{q,n}\|_{\dot B^{-1}_{\infty,\infty}}d\tau
\\&
+\int_0^te^{-C\|\tilde u\|_{L^1_{\tau}(\dot B^1_{\infty,1})}}
\|\tilde\omega_{q,n}\|_{\dot B^{-1}_{\infty,\infty}}\|u\|_{\dot B^{1}_{2,1}}d\tau\Bigr\}
e^{\|u\|_{L^1_t(L^\infty)}+\|u\|_{L^1_t(\dot B^1_{2,1})}}.
\end{aligned}
$$ 
This along with Proposition \ref{proposition}, Corollary \ref{CM}
and inequality \eqref{LT}, ensures that
$$
\begin{aligned}
e^{-C\|\tilde u\|_{L^1_t(\dot B^1_{\infty,1})}}
||(x,y)\tilde \omega_{q,n}(t)||_{\dot B^{-1}_{\infty,\infty}}
&\le
C_02^{-q}e^{C_0U(t)}\exp(C_0e^{C_0t})
\\&
\times
\Bigl\{\|\dot\Delta_q(1,x,y)\omega^0_{n}\|_{L^\infty}+
c_q(\|\omega^0_n\|_{\dot B^0_{2,1}}+\|\omega^0_{z,n}\|_{\dot B^0_{1,1}})\Bigr\}.
\end{aligned}
$$ 
This completes the proof of Proposition \ref{PC}.
\end{proof}
For conclude remnant  to controlled $\tilde \omega_{q,n}$ in $\dot B^1_{2,1}$ and
$\dot B^{-1}_{2,\infty}$ (see Remark \ref{LLT}).
\begin{proposition}\label{RD}
There exists $C_0$ is a constant depending on $u^0$ and $c_q\in\ell^1(\ZZ),$ such that.\\
i) For every $(q,n)\in \ZZ^2,$ we have
$$
||\tilde \omega_{q,n}(t)||_{L^{2}}+||(1,x,y)\tilde \omega_{q,n}(t)||_{L^{\infty}}
\le
C_0\bigl(||\dot \Delta _q\omega^0_n||_{L^2}
+c_q\|\omega^0_{z,n}\|_{\dot B^0_{1,1}}\bigr)e^{C_0t}.
$$ 
ii) For every $(j,q,n)\in \ZZ^3,$ we have
$$
\begin{aligned}
||\dot \Delta_j\tilde \omega_{q,n}(t)||_{L^2}
&\leq 
C_02^{-|j-q|}e^{C_0U(t)}\exp(C_0e^{C_0t})
\Bigl(||\dot \Delta _q \omega^0_n||_{L^2}
+c_q\|\omega^0_{z,n}\|_{\dot B^0_{1,1}}\Bigr).
\end{aligned}
$$
with $U(t):=||\tilde u||_{L_t^1(\dot B_{\infty,1}^1)}+||u||_{L_t^1(\dot B_{2,1}^1)}.$
\end{proposition}
\begin{proof}
It follows from \eqref{VH} that
$$
\left\{ \begin{array}{lll}
\partial_t \tilde \omega_{1,q,n}+(\tilde u\cdot\nabla_h)\tilde \omega_{1,q,n}
=-\tilde \omega_{z,q,n}\tilde u_2\\
\partial_t \tilde \omega_{2,q,n}+(\tilde u\cdot\nabla_h)\tilde \omega_{2,q,n}
=\tilde \omega_{z,q,n}\tilde u_1\\
\partial_t \tilde \omega_{z,q,n}+(\tilde u\cdot\nabla_h)\tilde \omega_{z,q,n}
=0\\
\tilde \omega_{1,q,n|t=0}=\tilde\omega_{q,n}(0).
\end{array}\right.
$$
Taking $L^2$ inner product of the above system with $\tilde\omega_{q,n}$ gives
$$
\frac{1}{2}\frac{d}{dt}\|\tilde\omega_{q,n}\|_{L^2}^2
\le
\|\tilde\omega_{q,n}\|_{L^2}\|\tilde\omega_{z,q,n}\|_{L^2}\|\tilde u\|_{L^\infty}
\le
\|\dot\Delta_q\omega_{z,n}^0\|_{L^2}\|\tilde\omega_{q,n}\|_{L^2}\|\tilde u\|_{L^\infty}.
$$
Applying Gronwall's inequality and using Propositions \ref{henda}, \ref{proposition} gives rise to
$$
||\tilde \omega_{q,n}(t)||_{L^{2}}
\le
C_0\bigl(||\dot \Delta _q\omega^0_n||_{L^2}
+c_q\|\omega^0_{z,n}\|_{\dot B^0_{1,1}}\bigr)e^{C_0t}.
$$ 
Thanks to Corollary \ref{saoussen}, we obtain 
$$
\begin{aligned}
e^{-C\|\tilde u\|_{L^1_t(\dot B^1_{2,1})}}||\tilde\omega_{q,n}(t)||_{\dot B_{2,1}^1}
&\lesssim 
||\tilde \omega_{q,n}(0)||_{\dot B_{2,1}^1}+\int_0^t 
e^{-C\|\tilde u\|_{L^1_{\tau}(\dot B^1_{\infty,1})}}
||\tilde\omega_{z,q,n}\tilde u||_{\dot B_{2,1}^1}d\tau
\\&
\lesssim
||\tilde \omega_{q,n}(0)||_{\dot B_{2,1}^1}+\int_0^t 
e^{-C\|\tilde u\|_{L^1_{\tau}(\dot B^1_{\infty,1})}}
||\tilde\omega_{z,q,n}||_{\dot B_{2,1}^1}\|\tilde u||_{L^\infty}d\tau
\\&
+\int_0^t 
e^{-C\|\tilde u\|_{L^1_{\tau}(\dot B^1_{\infty,1})}}
||\tilde\omega_{z,q,n}||_{L^2}\|\tilde u||_{\dot B_{\infty,1}^1}d\tau.
\end{aligned}
$$
As $\tilde\omega_{z,q,n}$ satisfies
$$
e^{-C\|\tilde u\|_{L^1_{\tau}(\dot B^1_{\infty,1})}}
||\tilde\omega_{z,q,n}||_{\dot B_{2,1}^1}
\le
\|\dot\Delta_q\omega_{z,n}^0\|_{\dot B^1_{2,1}}
\lesssim
2^q\|\dot\Delta_q\omega^0_{z,n}\|_{L^2}
$$
and
$$
||\tilde\omega_{z,q,n}||_{L^2}
\le
||\dot\Delta_q\omega_{z,n}^0||_{L^2}
\lesssim
2^qc_q\|\omega_{z,n}^0\|_{\dot B^0_{1,1}}.
$$
Then from Propositions \ref{proposition} and \ref{henda} we find that for $(q,n)\in\ZZ^2$
$$
||\tilde\omega_{q,n}(t)||_{\dot B_{2,1}^1}
\le
C_02^q\bigl(||\dot \Delta _q\omega^0_n||_{L^2}
+c_q\|\omega^0_{z,n}\|_{\dot B^0_{1,1}}\bigr)
e^{C_0(t+\|\tilde u\|_{L^1_{\tau}(\dot B^1_{\infty,1})})}.
$$ 
By a similar proof of the previous inequality, we deduce
$$
||\tilde\omega_{q,n}(t)||_{\dot B_{2,\infty}^{-1}}
\le
C_02^{-q}\bigl(||\dot \Delta _q\omega^0_n||_{L^2}
+c_q\|\omega^0_{z,n}\|_{\dot B^0_{1,1}}\bigr)
e^{C(\|\tilde u\|_{L^1_{\tau}(\dot B^1_{2,1})}+\|\tilde u\|_{L^1_{\tau}(\dot B^1_{\infty,1})})}
\exp(C_0e^{c_0t}).
$$ 
This completes the proof.
\end{proof}
So in conclusion, we obtain the following corollary.
\begin{corollary}\label{NB}
For every $(j,q,n)\in \ZZ^3,$ we have
$$
\begin{aligned}
||\tilde \omega_{q,n}(t)||_{L^{2}}+||\tilde \omega_{q,n}(t)||_{L^{\infty}}
\le
C_0\bigl(||\dot \Delta _q\{(1,x,y)\omega^0_n\}||_{L^{\infty}}
+c_q\|\omega^0_{n}\|_{\dot B^0_{2,1}}
+c_q\|\omega^0_{z,n}\|_{\dot B^0_{1,1}}\bigr)e^{C_0t}
\end{aligned}
$$ 
and
$$
\begin{aligned}
||\dot \Delta_j\tilde \omega_{q,n}(t)||_{L^2}
&+||\dot \Delta_j\{(1,x,y)\tilde \omega_{q,n}\}(t)||_{L^{\infty}}
\leq 
C_02^{-|j-q|}e^{C_0U(t)}\exp(C_0e^{C_0t})
\\&
\times
\Bigl(||\dot \Delta_q\{(1,x,y)\omega^0_n\}||_{L^{\infty}}+||\dot \Delta _q \omega^0_n||_{L^2}
+c_q\|\omega^0_{n}\|_{\dot B^0_{2,1}}+c_q\|\omega^0_{z,n}\|_{\dot B^0_{1,1}}
\Bigr),
\end{aligned}
$$
with $U(t):=||\tilde u||_{L_t^1(\dot B_{\infty,1}^1)}+||u||_{L_t^1(\dot B_{2,1}^1)}$
and $c_q\in\ell^1(\ZZ).$
\end{corollary}
\begin{proposition}\label{p}
The solution of ${\rm(E)}$ with initial data $(1,x,y)u^0\in L^2$, such that 
$\omega^0\in\mathscr{\dot B}^{0}_{2,1},$ 
$(1,x,y)\omega^0\in\mathscr{\dot B}^{0}_{\infty,1}$ and
$\omega^0_z\in\mathscr{\dot B}^{0}_{1,1}$ satisfies for every $t\in\RR_+$
$$
\|(1,x,y)u(t)\|_{L^2}\le C_0e^{C_0t}
$$
and
$$
||\omega_z(t)||_{\mathscr{\dot B}^{0}_{1,1}}+
||\omega(t)||_{\mathscr{\dot B}^{0}_{2,1}}
+||(1,x,y)\omega(t)||_{\mathscr{\dot B}^{0}_{\infty,1}}\leq C_0 e^{exp(e^{C_0t})}
$$
where $C_0$ depends on the norms of $u^0.$
\end{proposition}
\begin{proof}
Inequality \eqref{VS} implies the first estimate.
 Note that for any fixed integer $N,$ one has 
\begin{eqnarray}\label{ilyes}
||\omega_n(t)||_{\dot B_{2,1}^0}&+&||(1,x,y)\omega_n(t)||_{\dot B_{\infty,1}^0}
\leq
\sum_j||\dot \Delta_j\sum_q(1,x,y)\tilde \omega_{q,n}(t)||_{L^{\infty}}
+\sum_j||\dot \Delta_j\sum_q\tilde \omega_{q,n}(t)||_{L^{2}}
\nonumber\\
&\leq& \sum_{|j-q|\geq N}||\dot \Delta_j\{(1,x,y)\tilde \omega_{q,n}\}(t)||_{L^{\infty}}
+\sum_{|j-q|\geq N}||\dot \Delta_j\tilde \omega_{q,n}(t)||_{L^{2}}\nonumber\\
&+& \sum_{|j-q|<N}||\dot \Delta_j\{(1,x,y)\tilde \omega_{q,n}\}(t)||_{L^{\infty}}
+\sum_{|j-q|<N}||\dot \Delta_j\tilde \omega_{q,n}(t)||_{L^{2}}\nonumber\\
&:=&f_n+g_n
\end{eqnarray}
Applying Corollary \ref{NB} gives
\begin{eqnarray}\label{skander}
f_n&\le&
C_02^{-N}e^{C_0U(t)}\exp(C_0e^{C_0t})
\Bigl(||(1,x,y)\omega^0_n||_{\dot B^0_{\infty,1}}
+\|\omega^0_{n}\|_{\dot B^0_{2,1}}+\|\omega^0_{z,n}\|_{\dot B^0_{1,1}}
\Bigr)
\end{eqnarray}
and
\begin{equation}\label{asma}
g_n\le
C_0N\bigl(||(1,x,y)\omega^0_n||_{\dot B^0_{\infty,1}}
+\|\omega^0_{n}\|_{\dot B^0_{2,1}}
+\|\omega^0_{z,n}\|_{\dot B^0_{1,1}}\bigr)e^{C_0t}.
\end{equation}
Combining this estimate with \eqref{ilyes}, \eqref{skander} and \eqref{asma}, we obtain
$$
\begin{aligned}
||\omega_n(t)||_{\dot B_{2,1}^0}&+||(1,x,y)\omega_n(t)||_{\dot B_{\infty,1}^0}
\lesssim 
N\bigl(||(1,x,y)\omega^0_n||_{\dot B^0_{\infty,1}}
+\|\omega^0_{n}\|_{\dot B^0_{2,1}}
+\|\omega^0_{z,n}\|_{\dot B^0_{1,1}}\bigr)e^{C_0t}
\\&
+2^{-N}e^{C_0U(t)}\exp(C_0e^{C_0t})
\Bigl(||(1,x,y)\omega^0_n||_{\dot B^0_{\infty,1}}
+\|\omega^0_{n}\|_{\dot B^0_{2,1}}+\|\omega^0_{z,n}\|_{\dot B^0_{1,1}}
\Bigr).
\end{aligned}
$$
Choosing the integer $N$ so that  
$$
N\approx U(t),
$$
leads to
$$
\begin{aligned}
||\omega_n(t)||_{\dot B_{2,1}^0}+||(1,x,y)\omega_n(t)||_{\dot B_{\infty,1}^0}
&\lesssim 
U(t)\exp(C_0e^{C_0t})
\\&
\times\Bigl(||(1,x,y)\omega^0_n||_{\dot B^0_{\infty,1}}
+\|\omega^0_{n}\|_{\dot B^0_{2,1}}+\|\omega^0_{z,n}\|_{\dot B^0_{1,1}}
\Bigr).
\end{aligned}
$$
Lemmas \ref{LP}, \ref{RF} and Remark \ref{LLT}
on the other hand, we have
$$
\begin{aligned}
\sum_{n\in\ZZ}||\omega_n(t)||_{\dot B_{2,1}^0}
&+\sum_{n\in\ZZ}||(1,x,y)\omega_n(t)||_{\dot B_{\infty,1}^0}
\lesssim
\bigl(||(1,x,y)\omega^0||_{\mathscr{\dot B}^0_{\infty,1}}
+\|\omega^0\|_{\mathscr{\dot B}^0_{2,1}}
+\|\omega^0_{z}\|_{\mathscr{\dot B}^0_{1,1}}
\bigr) 
\\&
\times
\int_0^t\bigl(\|u\|_{L^2(\RR^2\times]-\pi,\pi[)}+
\|(1,x,y)\omega\|_{\mathscr{\dot B}^{0}_{\infty,1}}
+\|\omega\|_{\mathscr{\dot B}^{0}_{2,1}}\bigr)d\tau
\exp(C_0e^{C_0t}),
\end{aligned}
$$
hence we obtain by Gronwall's inequality
$$
||\omega(t)||_{\mathscr{\dot B}^{0}_{2,1}}
+||(1,x,y)\omega(t)||_{\mathscr{\dot B}^{0}_{\infty,1}}\leq C_0 e^{exp(e^{C_0t})}.
$$
As $\omega_z$ verifies 
$$
\partial_t\omega_z+(\tilde u\cdot\nabla_h)\omega_z=0,
$$
we obtain by Corollary \ref{saoussen}
$$
\|\omega_z\||_{\mathscr{\dot B}^{0}_{1,1}}\leq C_0 e^{exp(e^{C_0t})}.
$$
This finishes the proof.
\end{proof}
\subsection{Existence and uniqueness}
The proof of existence of a solution is performed in a standard manner.
We begin by solving an approximate problem, we are going to use Friedrich's  method, which consists to approximation of  system ${\rm(E)}$ by a truncation in the space of the frequencies. Let us define then the operator
$$
J_{\ell,k}u=\sum_{|n|\le k}e^{inz}1_{r\le\ell}\mathcal{F}^vu(n).
$$
Let us consider the  approximate equation
$$
\partial_tu_{\ell,k}+J_{\ell,k}\bigl(J_{\ell,k}u_{\ell,k}\cdot\nabla)J_{\ell,k}u_{\ell,k}
=J_{\ell,k}\mathcal{Q}(J_{\ell,k}u_{\ell,k},u_{\ell,k})
$$
with
$$
\mathcal{Q}(u,v)=\sum_{i,j}\partial_i\partial_j(-\Delta)^{-1}(u^iu^j).
$$
Later we prove that the solutions are uniformly
bounded. The last step consists in studying the convergence to a solution of the
initial equation. So we prove the local existence for regular data
(for more details see \cite{rd,  dutrifoy}).
In critical spaces one can follow Park's approach
in \cite{park}. To prove that the solution associated to all helicoidal and enough smooth initial data $u^0$, is helicoidal, it suffices to use a method due to  X. Saint Raymond \cite{Saint}. In fact, it's clear that the first condition of  Definition \ref{he} is satisfied. Concerning the second condition, we have
$$
\partial_t\{u(re_\theta+ke_z)\}+(u\cdot\nabla)\{u(re_\theta+ke_z)\}
=\{(u\cdot\nabla)(re_\theta+ke_z)\}u-\nabla\Pi(re_\theta+ke_z)
$$
i.e;
$$
\partial_t\{u(re_\theta+ke_z)\}+(u\cdot\nabla)\{u(re_\theta+ke_z)\}
=0.
$$ 
To prove the uniqueness simply to controlled the difference of two solutions in
$L^2(\RR^2\times]-\pi,\pi[).$
\section*{Appendix}
To prove main theorem, we need some inequalities. 
\begin{proposition}\label{Commutateur}
Let $(r,p)\in[1, +\infty]^2,$ $f\in\dot B^s_{p,r},$ $v$ be a smooth divergence free vector field and 
$$
[v\cdot\nabla,\dot\Delta_q]f=(v\cdot\nabla)\dot\Delta_qf
-\dot\Delta_q((v\cdot\nabla)f).
$$
Then there hold
\begin{itemize}
\item[(i)] If $s=-1$
$$
\sup_q2^{-q}\|[v\cdot\nabla,\dot\Delta_q]f\|_{L^p}
\lesssim
\|\nabla v\|_{\dot B^0_{\infty,1}}\|f\|_{\dot B^{-1}_{p,\infty}}
$$
\item[(ii)] If $-1<s<1$
$$
\big(\sum_q2^{qsr}\|[v\cdot\nabla,\dot\Delta_q]f\|_{L^p}^r\big)^{1\over r}
\lesssim
\|\nabla v\|_{L^\infty}\|f\|_{\dot B^s_{p,r}}.
$$
\item[(iii)] If $s=1$
$$
\sum_q2^q\|[v\cdot\nabla,\dot\Delta_q]f\|_{L^p}
\lesssim
\|\nabla v\|_{\dot B^0_{\infty,1}}\|f\|_{\dot B^1_{p,1}}.
$$
\end{itemize}
If $f=\rot\,u,$ then the second point holds true for $s\in[1,\infty[.$
\end{proposition}
\begin{proof}
Thanks to Bony's decomposition, we write
$$
[v\cdot\nabla,\dot\Delta_q] f =[T_{v^j},\dot\Delta_q]\partial_jf
+T_{\partial_j\dot\Delta_qf}v^j
-\dot\Delta_qT_{\partial_jf}v^j
+R(v^j,\dot\Delta_q\partial_jf)-\dot\Delta_qR(v^j,\partial_jf).
$$
For every $s\in\RR,$ by a classical inequality about commutators we have
(see for example \cite{chemin}) 
$$
\begin{aligned}
\big\{\sum_q 2^{qsr}\|[T_{v^j},\dot\Delta_q]
\partial_jf\|_{L^p}^r\big\}^{1\over r}
&\le
\big\{\sum_q 2^{qsr}\|[T_{v^j},\dot\Delta_q]
\partial_jf\|_{L^p}^r\big\}^{1\over r}
\\&
\lesssim
\|\nabla v\|_{L^\infty}\|f\|_{\dot B^s_{p,r}}.
\end{aligned}
$$
For the paraproduct term, we have
$$
T_{\partial_j\dot\Delta_qf}v^j=\sum_{q\le q'}S_{q'-1}\dot\Delta_q\partial_jf\dot\Delta_{q'}v^j,
$$
Applying Bernstein and H\"older inequalities leads to
$$
\begin{aligned}
\Big\|\sum_{q\le q'}\dot S_{q'-1}\dot\Delta_q\partial_jf\dot\Delta_{q'}v^j
\Big\|_{L^p}
&\lesssim
\sum_{q\le q'}
\|\dot\Delta_{q'}v^j\|_{L^\infty}\|\|\dot\Delta_q\partial_jf\|_{L^p}
\\&
\lesssim
\sum_{q\le q'}2^{q-q'}
\|\dot\Delta_{q'}\nabla v\|_{L^\infty}\|\dot\Delta_qf\|_{L^p}.
\end{aligned}
$$
Therefore
$$
\big\{\sum_q 2^{qsr}\|T_{\partial_j\dot\Delta_qf}
v^j\|_{L^p}^r\big\}^{1\over r}
\lesssim
\|\nabla v\|_{\dot B^0_{\infty,\infty}}\|f\|_{\dot B^s_{p,r}}\qquad\forall\,s\in\RR.
$$
Concerning the term $\dot\Delta_qT_{\partial_jf}v^j,$ we have
$$
\dot\Delta_qT_{\partial_jf}v^j
=
\sum_{|q-q'|\le4}\Delta_q(S_{q'-1}\partial_jf\dot\Delta_{q'}v^j).
$$
From the definition of $\dot S_{q'-1}$ and applying Bernstein inequality, we obtain
$$
2^{qs}\|\dot\Delta_qT_{\partial_jf}v^j\|_{L^p}
\lesssim
\sum_{|q-q'|\le4}\|\dot\Delta_{q'}\nabla v\|_{L^\infty}
\sum_{k\le q'-2}2^{(k-q')(1-s)}2^{ks}\|\Delta_kf\|_{L^p}.
$$
Thus
$$
\big(\sum_q 2^{qsr}\|\dot\Delta_qT_{\partial_jf}v^j
\|_{L^p}^r\big)^{1\over r}
\lesssim
\|\nabla v\|_{\dot B^0_{\infty,\infty}}
\|f\|_{\dot B^s_{p,r}} \quad\quad\forall\;s<1
$$
and for $s=1$
$$
\sum_q 2^q\|\dot\Delta_qT_{\partial_jf}v^j\|_{L^p}
\lesssim
\|\nabla v\|_{\dot B^0_{\infty,1}}\|f\|_{\dot B^1_{p,1}}.
$$
Finally divergence free of $v,$ implies
$$
\begin{aligned}
 R(v^j,\dot\Delta_q\partial_jf)-\dot\Delta_qR(v^j,\partial_jf)
&=\sum_{q'\geq q-3}\dot\Delta_{q'}v^j\widetilde{\dot\Delta}_{q'}
\dot\Delta_q\partial_jf-\dot\Delta_q\partial_j\big(\dot\Delta_{q'}v^j
\widetilde{\dot\Delta}_{q'}f\big)
:=\hbox{I}_q.
\end{aligned}
$$
From H\"older and  Bernstein inequalities, we deduce
$$
\big(\sum_q 2^{qsr}\|\hbox{I}_q\|_{L^p}^r\big)^{1\over r}
\lesssim
\|\nabla v\|_{\dot B^0_{\infty,\infty}}
\|f\|_{\dot B^s_{p,r}} \quad\quad\hbox{if}\quad s>-1,
$$
$$
\sup_q2^{-q}\|\hbox{I}_q\|_{L^p}
\lesssim
\|\nabla v\|_{\dot B^0_{\infty,1}}\|f\|_{\dot B^{-1}_{p,\infty}}
\quad\quad\hbox{if}\quad s=-1.
$$
The proof is now achieved.
 \end{proof}
An immediate corollary of the above lemma is.
\begin{corollary}\label{saoussen}
Let $s\in ]-1,1[$, $(p,r)\in\ [1,\infty]^2$ and $u$ be a smooth divergence
free vector field. Let $f$ a smooth solution of the transport equation.
$$
\partial_t f+ u.\nabla f=g,\qquad f_{|t=0}=f_0,
$$
such that $f_0\in \dot B_{p,r}^s(\RR^2)$ and $g\in \ L^1_{loc}(\RR_+;\, \dot B_{p,r}^s).$ Then
\begin{equation}\label{tr}
||f(t)||_{\dot B_{p,r}^s}
\leq Ce^{CU(t)} 
\bigl(||f_0||_{\dot B_{p,r}^s}+\int_0^t e^{-CU(\tau)}||g(\tau)||_{ \dot B_{p,r}^s}d\tau\bigr)
\qquad\forall\, t\in \RR_+.
\end{equation}
where $U(t)=\int_0^t ||\nabla u(\tau)||_{L^{\infty}} d\tau$ and $C$ constant depends on $s$.\\
The above estimate holds also true in the limit cases :
$$
s=-1,\ r=\infty,\ p\in\ [1,\infty]\ \mbox{or}\ s=1,\ r=1,\ p\in \ [1,\infty],
$$
provided that we change $U(t)$ by $U_1(t):= ||u||_{L_t^1\dot B_{\infty,1}^1}.$\\
In addition if $f=\rot\,u,$ then the above estimate \eqref{tr} holds true for all $s\in[1,+\infty[.$
\end{corollary}
\begin{proposition}\label{proposition}
Let  $f\in\mathcal{S}'(\RR^2),$ then\\
i)
$$
x_1\dot \Delta_jf=\dot \Delta_j (y_1f)+\tilde h^1_j\ast f
\qquad\mbox{and}\qquad
x_2\dot \Delta_jf=\dot \Delta_j (y_2f)+\tilde h^2_j\ast f
 $$
with 
$$
\tilde{h}^1_j(x)=2^{-j}x_{1}h(2^jx)\qquad\mbox{and}\qquad \tilde{h}^2(x)=2^{-j}x_{2}h(2^jx).
$$ 
In addition
$$
\dot\Delta_q(x_1\dot \Delta_j \omega)
=\dot\Delta_q(x_2\dot \Delta_j \omega)
=0,\qquad\mbox{for}\quad |j-q|\geq5.
$$
ii) If $f \in\dot B^0_{2,1}\cap \dot B^0_{1,1},$ then
$$
\|x_1\dot \Delta_jf-\dot \Delta_j (y_1f)\|_{L^\infty}
+\|x_2\dot \Delta_jf-\dot \Delta_j (y_2f)\|_{L^\infty}
\lesssim
c_j
\left\{\begin{array}{l}
\displaystyle\|f\|_{\dot B^{0}_{2,1}},\\
\displaystyle2^j\|f\|_{\dot B^{0}_{1,1}},
\end{array}
\right.
$$
$$
\|x_1\dot \Delta_jf-\dot \Delta_j (y_1f)\|_{\dot B^1_{\infty,1}}
+\|x_2\dot \Delta_jf-\dot \Delta_j (y_2f)\|_{\dot B^1_{\infty,1}}
\lesssim
c_j2^j\|f\|_{\dot B^{0}_{2,1}},
$$
$$
\|x_1\dot \Delta_jf-\dot \Delta_j (y_1f)\|_{\dot B^{-1}_{\infty,1}}
+\|x_2\dot \Delta_jf-\dot \Delta_j (y_2f)\|_{\dot B^{-1}_{\infty,1}}
\lesssim
c_j2^{-j}\|f\|_{\dot B^{0}_{2,1}},
$$
$$
\|x_1\dot \Delta_jf-\dot \Delta_j (y_1f)\|_{L^2}
+\|x_2\dot \Delta_jf-\dot \Delta_j (y_2f)\|_{L^2}
\lesssim
c_j
\|f\|_{\dot B^{0}_{1,1}},
$$
$$
\|x_1\dot \Delta_jf-\dot \Delta_j (y_1f)\|_{\dot B^1_{2,1}}
+\|x_2\dot \Delta_jf-\dot \Delta_j (y_2f)\|_{\dot B^1_{2,1}}
\lesssim
c_j2^j\|f\|_{\dot B^{0}_{1,1}},
$$
$$
\|x_1\dot \Delta_jf-\dot \Delta_j (y_1f)\|_{\dot B^{-1}_{2,1}}
+\|x_2\dot \Delta_jf-\dot \Delta_j (y_2f)\|_{\dot B^{-1}_{2,1}}
\lesssim
c_j2^{-j}\|f\|_{\dot B^{0}_{1,1}}
$$
with $c_j\in\ell^1(\ZZ).$
\end{proposition}
\begin{proof}
i) We write by definition
\begin{eqnarray*}
x_1\dot \Delta_jf(x)-\dot \Delta_j (y_1f)(x)&=&2^{2j}\int_{\RR^2}h(2^j(x-y))(x_{1}-y_{1})f(y)dy \\
 &=&2^{2j}\tilde{h}^1_j\star f(x),
\end{eqnarray*}
with $\tilde{h}^1_j(x)=2^{-j}x_{1}h(2^jx).$ This complete the proof of i)\\
ii) Now we claim that for every $f\in\mathcal{S}'$ we have
$$
2^{2j}\tilde{h}(2^j\cdot)\star f=\sum_{|j-k|\leq 1}2^{2j}\tilde{h}(2^j\cdot)\star\dot\Delta_{k}f.  
$$
Indeed, we have  $\widehat{\tilde h}(\xi)=i\partial_{\xi_{1}}\widehat{h}(\xi)=i\partial_{\xi_{1}}\varphi(\xi).  $ It follows that $supp\, \widehat{\tilde h}\subset supp\, \varphi.  $ So we get 
$$
2^{2j}\tilde{h}(2^j\cdot)\star\dot\Delta_{k}f=0,\qquad\mbox{for $|j-k|\geq 2.$}  
$$
This leads to
\begin{eqnarray*}
\nonumber\sum_{j\in\ZZ}\|x_1\dot \Delta_jf-\dot \Delta_j (y_1f)\|_{L^\infty}
&\lesssim&
\sum_{|j-k|\leq 1}2^{k-j}2^{-k}\|\Delta_{k}f\|_{L^\infty}
\\
&\lesssim& 
\|f\|_{\dot B_{\infty,1}^{-1}}
\\
&\lesssim&
\|f\|_{\dot B_{2,1}^{0}}.  
\end{eqnarray*}
Similar for same inequalities. The proof is now achieved. 
\end{proof}
We follow the same proof of the previous proposition and we use the fact that
$$
x_ix_j-y_iy_j=(x_i-y_i)(x_j-y_j)+(x_i-y_i)y_j+(x_j-y_j)y_i,
$$
we obtain the following corollary. 
\begin{corollary}\label{CM}
If $(f,y_if) \in\dot B^0_{1,1}\times\dot B^0_{2,1},$ then for $1\le i,j\le2,$ we have
$$
\|x_ix_j\dot \Delta_jf-\dot \Delta_j (y_iy_jf)\|_{L^\infty}
\lesssim
c_j
\bigl(\|f\|_{\dot B^{0}_{1,1}}+\|y_1f\|_{\dot B^{0}_{2,1}}+\|y_2f\|_{\dot B^{0}_{2,1}}\bigr),
$$
$$
\|x_ix_j\dot \Delta_jf-\dot \Delta_j (y_iy_jf)\|_{\dot B^1_{\infty,1}}
\lesssim
c_j2^j\bigl(\|f\|_{\dot B^{0}_{1,1}}+\|y_1f\|_{\dot B^{0}_{2,1}}+\|y_2f\|_{\dot B^{0}_{2,1}}\bigr)
$$
and
$$
\|x_ix_j\dot \Delta_jf-\dot \Delta_j (y_iy_jf)\|_{\dot B^{-1}_{\infty,1}}
\lesssim
c_j2^{-j}\bigl(\|f\|_{\dot B^{0}_{1,1}}+\|y_1f\|_{\dot B^{0}_{2,1}}
+\|y_2f\|_{\dot B^{0}_{2,1}}\bigr)
$$
with $c_j\in\ell^1(\ZZ).$
\end{corollary}

\end{document}